\documentclass[10pt]{amsart}
\usepackage{CJK}
\usepackage{indentfirst}
\usepackage{amscd}
\usepackage{amsmath, amssymb}
\usepackage{amsfonts}
\usepackage{enumerate}
\usepackage{color}
\usepackage[numbers,sort&compress]{natbib}

\newcommand{\ve}{\varepsilon}
\newcommand{\ddbar}{\sqrt{-1} \partial \overline{\partial}}
\newcommand{\ol}{\overline}
\newcommand{\MA}{Monge-Amp\`{e}re}

\newcommand{\ri}{\rightarrow}

\newcommand{\ov}[1]{\overline{#1}}
\newcommand{\de}{\partial}
\newcommand{\dbar}{\overline{\partial}}
\newcommand{\ti}{\tilde}

\newtheorem*{claim*}{Claim}

\begin{document}
	\newcounter{theor}
	\setcounter{theor}{1}
	\newtheorem{claim}{Claim}
	\newtheorem{theorem}{Theorem}[section]
	\newtheorem{lemma}[theorem]{Lemma}
	\newtheorem{corollary}[theorem]{Corollary}
	\newtheorem{proposition}[theorem]{Proposition}
	\newtheorem{prop}{Proposition}[section]
	\newtheorem{question}{question}[section]
	\newtheorem{defn}{Definition}[section]
	\newtheorem{remark}{Remark}[section]

	\numberwithin{equation}{section}
	
	\title[Fully nonlinear elliptic equations]{Fully nonlinear elliptic equations with gradient terms on compact almost Hermitian manifolds}
	
	
	\author[L. Huang]{Liding Huang}
	\address{Westlake Institute for Advanced Study (Westlake University), 18 Shilongshan Road, Cloud Town, Xihu District, Hangzhou, P.R. China}
	\email{huangliding@westlake.edu.cn}
	
	\author[J. Zhang]{Jiaogen Zhang}
	\address{School of Mathematical Sciences, University of Science and Technology of China, Hefei 230026, P.R. China}
	\email{zjgmath@ustc.edu.cn}

	\begin{abstract}
		We establish second order estimates for a general class of fully nonlinear elliptic equations with gradient terms on almost Hermitian manifolds including the deformed Hermitian-Yang-Mills equation and the equation in the proof of Gauduchon conjecture by Sz\'ekelyhidi-Tosatti-Weinkove. As applications, we also consider the existence of \MA~ equation and Hessian equations.
	\end{abstract}
	\maketitle

	\section{Introduction}\label{introduction}
	
	Let $(M,\chi,J)$ be a compact almost Hermitian manifold of real dimension $2n$, and $\omega$ is a fixed real $(1,1)$-form on $(M,J)$. For an arbitrary smooth function $u$, we write
	\[
	\omega_{u} := \omega+\ddbar{u}+Z(\partial u)
	=\omega+\frac{1}{2}(dJdu)^{(1,1)}+Z(\partial u),
	\]
where $Z(\partial u)$ denotes a smooth (1,1)-form depending on $\partial u$  linearly which will be specified later, and let $\mu(u)=(\mu_{1}(u),\ldots,\mu_{n}(u))$ be the eigenvalues of $\omega_{u}$ with respect to $\chi$. For the sake of notational convenience, we sometimes denote $\mu_{i}(u)$ by $\mu_{i}$ when no confusion will arise.
	In the current paper,  we consider the following fully nonlinear elliptic equations of
	the form
	\begin{equation}\label{nonlinear equation}
		F(\omega_{u}) = f(\mu_{1},\cdots,\mu_{n}) = h,
	\end{equation}
	where $h\in C^{\infty}(M)$ and $f$ is a smooth symmetric function in $\mathbb{R}^{n}$.
	
	
	\medskip
	The equation \eqref{nonlinear equation} covers many important elliptic equations in (almost) complex geometry. 
	A typical example of \eqref{nonlinear equation} is the following equation:
	\begin{equation}\label{n-1}
		\Big(\eta+\frac{1}{n-1}\big((\Delta_{\chi}u)\chi-\ddbar u\big)+W(\partial u)\Big)^{n} = e^{h}\chi^{n}.
	\end{equation}
	Here $\eta$ is an almost Hermitian metric, $\Delta_{\chi}$ denotes the canonical Laplacian operator of $\chi$ and $W=W(\partial u)$ is a Hermitian tensor that linearly depends on $\partial u$. On a Hermitian manifold,
	the equation \eqref{n-1} was introduced by Popovici \cite{Popovici15} and Tosatti-Weinkove \cite{TW19} independently. Recently, Sz\'ekelyhidi-Tosatti-Weinkove \cite{STW17}  confirmed the famous Gauduchon conjecture \cite{GP} by solving equation \eqref{n-1}.  When $W\equiv  0$, 
	the equation \eqref{n-1} is  the notion of  Monge-Amp\`ere equation for $(n-1)$-plurisubharmonic functions in pioneer works of Fu-Wang-Wu \cite{FWW10,FWW15}.

	\medskip
	The fully nonlinear elliptic equations with gradient terms on Hermitian manifolds have been researched extensively, we refer the reader to \cite{TW20,YR,GQR,FGZ,GN,YR20} and references therein.  On the framework of almost Hermitian manifolds,  to our knowledge most of  researches toward equation \eqref{nonlinear equation} are independent of $\partial u$. Inspired by these works, we shall consider the equation \eqref{nonlinear equation} on compact almost Hermitian manifolds.

	\medskip
	Let $\Gamma_{n}$ be the positive orthant in $\mathbb{R}^{n}$ and $\Gamma_{1}=\{\mu\in\mathbb{R}^{n}  \ :\ \sum_{i}\mu_{i}>0\}$.  In this paper, we always assume that $f$ is defined in a  symmetric open  and convex cone $\Gamma \subset \Gamma_{1}\subsetneq\mathbb{R}^{n}$ satisfying $\Gamma+\Gamma_{n}\subset \Gamma$, i.e. for any $\mu\in\Gamma $ and $\mu'\in \Gamma_{n}$, $\mu+\mu'\in\Gamma$. 
	Furthermore, modifying the setup of Sz\'ekelyhidi \cite{Szekelyhidi18}, suppose that
	\begin{enumerate}[(i)]\setlength{\itemsep}{1mm}
		\item $f_{i}=\frac{\de f}{\de\mu_{i}}>0$ for all $i$ and $f$ is concave  in $\Gamma$,
		\item $\sup_{\partial\Gamma}f<h<\sup_{\Gamma}f$,
		\item for any constant $\sup_{\partial\Gamma}f<\sigma< \sigma'<\sup_{\Gamma} f$, 
		there exists a positive constant $N$, depending only on $\sigma$ and $\sigma'$, such that $ \Gamma^{\sigma}+N \textbf{1}\subset \Gamma^{\sigma'}.$
	\end{enumerate}
	Here the sublevel set
	$\Gamma^{\sigma}=\{\mu\in\Gamma \ :\ f(\mu)>\sigma\}$
	is convex open for any $\sigma>\sup_{\partial \Gamma}f$ and
	\[
	\sup_{\de\Gamma}f = \sup_{\lambda'\in\de\Gamma}\limsup_{\lambda\rightarrow\lambda'}f(\lambda), \qquad \textbf{1}=(1,\ldots,1)\in\mathbb{R}^{n}.\]
	
	\begin{remark} {\normalfont The original setup in \cite{Szekelyhidi18} assume the   symmetric open and convex cone $\Gamma\subset\Gamma_1 \subsetneq\mathbb{R}^{n}$  satisfying
			\begin{equation}\label{assumption of cone}
				\text{the vertex of $\Gamma$ is at the origin and\ } \Gamma_{n}\subset\Gamma,
			\end{equation}
			$f$ is defined in $\Gamma$ 
			and satisfies (i), (ii) and 
			\begin{equation*}
				\textrm{(iii') \quad for any}~\sigma<\sup_{\Gamma} f~\textrm{and}~\mu\in \Gamma, \textrm{we have}~\lim_{t\rightarrow \infty}f(t\mu)>\sigma.
			\end{equation*}		
			Note that (iii') implies (iii) via \cite[Lemma 9]{Szekelyhidi18} if we further assume $\Gamma$  satisfies \eqref{assumption of cone}. Motivated by Mirror Symmetry and Mathematical Physics, Jacob-Yau \cite{JY} studied the  equation
			\begin{equation*}
				\sum_{i}\mathrm{arccot}\,\mu_{i}=\hat{\theta}, \quad \text{in}~\Gamma_{D}=\big\{\mu\in\mathbb{R}^{n}\ :\ 0<\sum_{i}\mathrm{arccot}\,\mu_{i}<\pi \big\}
			\end{equation*}
			for some real constant $\hat{\theta}$.  
			We can verify that this equation satisfies (iii) (see \S 2) while not for (iii'), and $\Gamma_{D}$ satisfies the assumption  $\Gamma+\Gamma_{n}\subset \Gamma$ rather than \eqref{assumption of cone}.}
	\end{remark}

	We have the following estimate:
	\begin{theorem}\label{main estimate}
		Let $(M,\chi,J)$ be a compact almost Hermitian manifold of real dimension $2n$.  Suppose that $u$ (resp. $\underline{u}$) is a smooth solution (resp. $\mathcal{C}$-subsolution) of (\ref{nonlinear equation}).  Then we have 
		\begin{equation*}
			\| u\|_{C^{2}(M,\chi)}\leq C(1+\sup_{M}|\partial u|_{\chi}^{2}),
		\end{equation*}
		where  $C$ is a constant depending on  $\underline{u}$, $h$, $Z$, $\omega$, 
		$f$, $\Gamma$ and $(M,\chi,J)$.
	\end{theorem}

	
	As an application, to begin, we solve the equation \eqref{n-1}. We have
	
	\begin{theorem}\label{n-1 MA equation}
		Let $(M,\chi,J)$ be a compact almost Hermitian manifold of real dimension $2n$ and $\eta$ be an almost Hermitian metric. There exists a unique pair $(u,c)\in C^{\infty}(M)\times\mathbb{R}$ such that
		\begin{equation}\label{sove n-1 MA equation}
			\begin{cases}
				\ \left(\eta+\frac{1}{n-1}\big((\Delta_{\chi} u)\chi-\ddbar u\big)+W(\partial u)\right)^{n} = e^{h+c}\chi^{n}, \\[2mm]
				\ \eta+\frac{1}{n-1}\big((\Delta_{\chi} u)\chi-\ddbar u\big)+W(\partial u)> 0, \\[1.5mm]
				\ \sup_{M}u = 0.
			\end{cases}   
		\end{equation}
	\end{theorem}
	For the complex Monge--Amp\`ere equation, Yau \cite{Yau78} solved it on a K\"ahler manifold and confirmed the famous Calabi's conjecture (see \cite{Calabi57}). In the non-K\"ahler setting, we refer the reader to \cite{Cherrier87,CTW19,GL10,Hanani96,TW10a,TW10b,ZZ11}. The classical complex Hessian equations also have been studied extensively, see \cite{DK17,Hou09,HMW10,Szekelyhidi18,Zhang17,CHZ17}. Similar to Theorem \ref{n-1 MA equation}, we can solve the complex Monge-Amp\`ere equation and complex Hessian equations  with gradient terms.

	\begin{theorem}\label{complex Hessian equation}
		Let $(M,\chi,J)$ be a compact almost Hermitian manifold of real dimension $2n$ and $\omega$ be a smooth $k$-positive real $(1,1)$-form. For any integer $1\leq k\leq n$, there exists a unique pair $(u,c)\in C^{\infty}(M)\times\mathbb{R}$ such that
	
	\begin{equation}\label{complex Hessian equation 1}
			\begin{cases}
				\ \omega_{u}^{k}\wedge \chi^{n-k}=e^{h+c}\chi^{n}, \\[1mm]
				\ \frac{\omega_{u}^{i}\wedge \chi^{n-i}}{\chi^{n}}>0, \quad i=1,2,\cdots,k, \\[0.5mm]
				\ \sup_{M}u = 0.
			\end{cases}
		\end{equation}
	\end{theorem}

	\medskip
	For the deformed Hermitian-Yang-Mills (dHYM) equation
	\begin{equation}\label{DHYM}
		\phi(\mu)=\sum_{i=1}^{n}\mathrm{arccot}\,\mu_{i}=h,\qquad h\in C^{\infty}(M),
	\end{equation}
	we say \eqref{DHYM} is hypercritical (resp. supercritical) if $h\in (0,\frac{\pi}{2})$ (resp. $h\in (0,\pi)$).  Jacob-Yau \cite{JY} showed  the existence of solution for dimension 2,  and for general dimensions when $(M,\chi)$ has  non-negative orthogonal bisectional curvature in the hypercritical phase setting. Pingali \cite{PV,PV1} obtained a solution when $n=3$. In general dimensions, the equation \eqref{DHYM}
	was solved by Collins-Jacob-Yau \cite{CJY}  under the existence of $\mathcal{C}$-subsolutions. The equation \eqref{DHYM} was also studied by Leung \cite{Leung97,Leung98} to seek vector bundles over a symplectic manifold. Recently,  Zhang  and the authors \cite{HZZ} provided a priori estimates  on compact almost Hermitian manifolds for the hypercritical case. It was researched by Lin \cite{Lin20} in the supercritical phase on compact Hermitian manifolds. 
	
	\medskip
	As a corollary, using Theorem \ref{main estimate}, we are also able to derive a priori estimates for \eqref{DHYM} in the supercritical case.
	
	\begin{corollary}\label{DHYM estimates}
Let $(M,\chi,J)$ be a compact almost Hermitian manifold of real dimension $2n$. 	Suppose that $u$ (resp.  $\underline{u}$) is the solution (resp. $\mathcal{C}$-subsolution) of  equation \eqref{DHYM}
		with $h\in \big(0,\pi-\delta\big]$ is a smooth function for a constant $\delta\in (0,\frac{\pi}{2})$. Then for each $\alpha\in(0,1)$, we have 
		\begin{equation*}
			\| u\|_{C^{2,\alpha}(M,\chi)}\leq C,
		\end{equation*}
		where $C$ is a constant depending on  $\alpha$, $\underline{u}$, $h$, $\omega$, $\delta$ and $(M,\chi,J)$. 
	\end{corollary}

	\medskip
	We now discuss the proof of Theorem \ref{main estimate}. The zero order estimate can be proved by adapting the arguments of \cite[Proposition 11]{Szekelyhidi18} and \cite[Proposition 3.1]{CTW19}, which are based on the method of B{\l}ocki \cite{Blocki05,Blocki11}. For the second order estimate,  following the idea of \cite{CHZ21, CM21,CTW19, Szekelyhidi18} and by some delicate calculations, the real Hessian $\nabla^{2}u$ can be controlled by the first gradient quadratically as follows:
	\begin{equation}\label{estimate introduction}
		\sup_{M}|\nabla^{2}u|_{\chi} \leq C\big(1+\sup_{M}|\partial u|_{\chi}^{2}\big).
	\end{equation}

	\medskip
	The paper is organized as follows. In \S \ref{preliminaries}, we will introduce some notations, and recall the definition and an important property of $\mathcal{C}$-subsolution. We also verify that  the dHYM equation satisfying the structural conditions. The zero order estimate will be established in \S \ref{zero order estimate}. In \S \ref{second order estimate},  we shall prove the estimate \eqref{estimate introduction}. To see this, we apply the maximum principle to the quantity involving the largest eigenvalue $\lambda_{1}$ of real Hessian $\nabla^{2}u$ with respect to $\chi$ of form 
	\begin{equation*}
		Q = \log \lambda_{1}+\varphi(|\rho|_{\chi}^{2})+\psi(|\partial u|_{\chi}^{2})+e^{-Au}.
	\end{equation*} 
In \S 3.3, we establish the second order estimate via the blowup argument and Liouville type theorem \cite[Theorem 20]{Szekelyhidi18} when equation \eqref{nonlinear equation} satisfying the structural conditions (i), (ii) and (iii'). Given this, we are able to prove Theorems \ref{n-1 MA equation}-\ref{complex Hessian equation} in \S 3.4. In \S 4 we will prove Corollary \ref{DHYM estimates} by using the maximum principle to establish the $C^{1}$ estimate for \eqref{DHYM}  which  also implies the $C^{2}$ estimate. 
	
	\medskip
	{\bf Acknowledgments.}
	We are very grateful to  Professor Xi Zhang for countless advice. We would like to thank Jianchun Chu for generous discussions. We are also grateful to Rirong Yuan for his  helpful suggestions.
	
	\section{Preliminaries}\label{preliminaries}
	\subsection{Notations}
	Suppose that $(M,\chi,J)$ is an almost Hermitian manifold of real dimension $2n$. As pointed in \cite[p.1954]{CTW19}, we can define $(p,q)$-forms and operators $\de$, $\dbar$ by using the almost complex structure $J$. Let $A^{1,1}(M)$ denote the set of smooth real (1,1)-forms on $(M,J)$.
	For any $u\in C^{\infty}(M)$, we see that 
$
\ddbar u= \frac{1}{2}(dJdu)^{(1,1)}
$
	is a real $(1,1)$-form in $A^{1,1}(M)$. In the sequel, we set
	\begin{equation*}
		\omega_{u}=\omega+\ddbar u+Z(\partial u),
	\end{equation*}
	where $Z(\partial u)$ is a real  $(1,1)$-form defined by
$
Z_{i\bar j}=Z_{i\bar j}^{p}u_{p}+\overline{Z_{i\bar j}^{ p}}u_{\bar p}.
$
	
	\medskip
	For any point $x_{0}\in M$, let $(e_{1},\cdots,e_{n})$ be a local unitary $(1,0)$-frame with respect to $\chi$ near $x_{0}$, and $\{\theta^{i}\}_{i=1}^{n}$ be its dual coframe. Then in the local chart we have
	\[
	\chi=\sqrt{-1}\delta_{ij}\theta^{i}\wedge\ov{\theta}^{j}.
	\]
	Suppose that
	\[
	\omega = \sqrt{-1}g_{i\ov{j}}\theta^{i}\wedge\ov{\theta}^{j}, \qquad
	\omega_{u} = \sqrt{-1}\ti{g}_{i\ov{j}}\theta^{i}\wedge\ov{\theta}^{j},
	\]
	as well as
	\begin{equation*}
		\begin{split}
			\ti{g}_{i\ov{j}} &= g_{i\ov{j}}+\de \dbar u(e_{i},\ov{e}_{j})+Z_{i\bar j}\\
			&= g_{i\ov{j}}+e_{i}\overline{e}_{j}(u)-[e_{i},\overline{e}_{j}]^{(0,1)}(u)+u_{p}Z_{i\bar j}^{p}+u_{\bar{p}}\overline{Z_{i\bar j}^{p}},	
		\end{split}
	\end{equation*}
	where $[e_{i}, \bar{e}_{j}]^{(0,1)}$ is the $(0,1)$ part of the Lie bracket $[e_{i}, \bar{e}_{j}]$.
	Define
	\begin{equation*}
		G^{i\overline{j}}=\frac{\partial F}{\de\ti{g}_{i\ov{j}}}, \qquad
		G^{i\overline{j},k\overline{l}}=\frac{\partial^{2}F}{\de\ti{g}_{i\ov{j}}\de\ti{g}_{k\ov{l}}}.
	\end{equation*}
	After making a unitary transformation, we may assume that  $\tilde{g}_{i\overline{j}}(x_{0})=\delta_{ij}\tilde{g}_{i\overline{i}}(x_{0})$.
	We denote $\tilde{g}_{i\overline{i}}(x_{0})$ by $\mu_{i}$.
	It is useful to order $\mu_{i}$ such that
	\begin{equation}\label{mu order}
		\mu_{1}\geq\mu_{2}\geq\cdots\geq\mu_{n}.
	\end{equation}
	At $x_{0}$, we have the expressions of $G^{i\ov{j}}$ and $G^{i\bar{k},j\bar{l}}$ (see e.g. \cite{Andrews94,Gerhardt96,Spruck05})
	\begin{equation}\label{second derive of F}
		G^{i\ov{j}} = \delta_{ij}f_{i},\qquad
		G^{i\bar{k},j\bar{l}}=f_{ij}\delta_{ik}\delta_{jl}
		+\frac{f_{i}-f_{k}}{\mu_{i}-\mu_{k}}
		(1-\delta_{ik})\delta_{il}\delta_{jk},
	\end{equation}
	where the quotient is interpreted as a limit if $\mu_{i}=\mu_{j}$. Using \eqref{mu order}, we obtain (see e.g. \cite{EH89,Spruck05})
	\[
	G^{1\ov{1}} \leq G^{2\ov{2}} \leq \cdots \leq G^{n\ov{n}}.
	\]
	On the other hand, the linearized operator of equation \eqref{nonlinear equation} is
	\begin{equation}\label{L}
		L(v)=G^{i\bar{j}}\Big(e_{i}\bar{e}_{j}(v)-[e_{i},\bar{e}_{j}]^{0,1}(v)+e_{p}(v)Z_{i\bar j}^{p}+\bar{e}_{p}(v)\overline{Z_{i\bar j}^{p}}\Big).
	\end{equation}

	\subsection{$\mathcal{C}$-subsolution}
	\begin{defn}[\cite{Szekelyhidi18}]\label{sub}
		We say that a function $\underline{u}\in C^{2}(M)$ is a $\mathcal{C}$-subsolution of \eqref{nonlinear equation} if at each point $x\in M$, the set
		\begin{equation*}
			\left\{\mu\in \Gamma\ :\ f(\mu)=h(x), \ \mu-\mu(\underline{u})\in \Gamma_{n}\right\}
		\end{equation*}
		is bounded. 
	\end{defn}
	
	\medskip
	By Definition \ref{sub}, for each $\mathcal{C}$-subsolution $\underline{u}$, there are constants $\delta, R>0$ depending only on $\underline{u}$, $(M,\chi,J)$, $f$ and $\Gamma$ such that 
	\begin{equation}\label{2..12}
		(\mu(\underline{u})-\delta\textbf{1}+\Gamma_{n})\cap \partial\Gamma^{h(x)}\subset B_{R}(0), \quad \forall\ x\in M,
	\end{equation}
	where $B_{R}(0)$ denotes the Euclidean ball with radius $R$ and center $0$.
	
	\medskip
	Similar to \cite{Guan14,Szekelyhidi18}, we have the following proposition:
	\begin{proposition}
		Suppose that $\sigma\in (\sup_{\partial \Gamma}f,\sup_{\Gamma}f)$ and $\mu\in\mathbb{R}^{n}$  satisfying 
		\begin{equation}\label{the set}
			(\mu-\delta\mathrm{\textbf{1}}+\Gamma_{n})\cap \partial\Gamma^{\sigma}\subset B_{R}(0)
		\end{equation}
		for some $\delta, R>0$. Then there exists a constant $\theta>0$ depending on $\delta$ and the set in \eqref{the set} such that for each $\mu'\in \partial \Gamma^{\sigma}$ and $|\mu'|>R$, we have either 
		$$\sum_{i}f_{i}(\mu')(\mu_{i}-\mu'_{i})>\theta\sum_{i}f_{i}(\mu'),$$
		or $f_{k}(\mu')>\theta\sum_{i}f_{i}(\mu')$ for each $k=1,2,\cdots,n$.
	\end{proposition}
	\begin{proof}
		The proof can be found in \cite[Proposition 5]{Szekelyhidi18}, we include it here for convenience to reader.	Set 
		$$A_{\delta}=\{v\in\Gamma\ :\ f(v)\leq \sigma \text{\ and\ } v-(\mu-\delta \textbf{1})\in\bar{\Gamma}_{n}\}.$$ It follows from  \eqref{the set} that $A_{\delta}$ is compact. For each $v\in A_{\delta}$, we define
		\[\mathcal{C}_{v}=\{w\in\mathbb{R}^{n}\ :\ v+tw\in (\mu-2\delta \textbf{1}+\Gamma_{n})\cap\partial \Gamma^{\sigma} \text{ \ for some\ } t>0 \}.\]
		Note that $f_{i}>0$ for all $i$. We conclude that
		\[\overline{(\mu-\delta \textbf{1}+\Gamma_{n})\cap\partial \Gamma^{\sigma}}\subsetneq (\mu-2\delta \textbf{1}+\Gamma_{n})\cap\partial \Gamma^{\sigma},\]
		which implies that $\mathcal{C}_{v}$ is strictly larger than $\Gamma_{n}$. Now we define the dual cone of $\mathcal{C}_{v}$ by
		\[\mathcal{C}_{v}^{*}=\{x\in\mathbb{R}^{n}\ :\ \left\langle x, y\right\rangle>0 \text{\ for all\ } y\in \mathcal{C}_{v}\}.\]
		We remark that $C_{v}\supsetneq \Gamma_{n}$  implies there exists a constant $\epsilon>0$ such that if $x=(x_{1},\cdots, x_{n})\in\mathcal{C}_{v}^{*}$, 
		\begin{equation}\label{dual cone property}
			x_{k}>\epsilon \text{\ for all \ } k .
		\end{equation}
		As $A_{\delta}$ compact, we can find a uniform constant $\epsilon$ such that \eqref{dual cone property} holds for all $v\in A_{\delta}$. Let $\mu'\in\partial \Gamma^{\sigma} $, $|\mu'|>R$ and $T_{\mu'}$ be the tangent plane to $\partial \Gamma^{\sigma}$ at $\mu'$. Now we split the proof into two cases:
		\medskip
		\begin{itemize}
			\item[Case 1.] Assume $T_{\mu'}\cap A_{\delta}\neq \emptyset$ and let $v\in T_{\mu'}\cap A_{\delta}$. Then the cone $v+\mathcal{C}_{v}$ lies above $T_{\mu'}$, i.e. $\left\langle x, n_{\mu'}\right\rangle>0$ for all $x\in \mathcal{C}_{v}$, where $n_{\mu'}$ is the inward pointing unit normal vector of $\partial \Gamma^{\sigma}$at $\mu'$. By the definition of $\mathcal{C}_{v}^{*}$,  we obtain $n_{\mu'}=Df(\mu')/|Df(\mu')|\in \mathcal{C}_{v}^{*}.$ It then follows \eqref{dual cone property} that for each $k=1,2,\cdots,n$
			\[f_{k}(\mu')>\epsilon|Df(\mu')|.\]
			\vspace{0.1cm}
			\item[Case 2.] We now assume $T_{\mu'}\cap A_{\delta}= \emptyset$, then $\text{dist}(\mu, T_{\mu'})>\delta$. Thus, $(\mu-\mu')\cdot n_{\mu'}>\delta,$ i.e.
			\[\sum_{i}f_{i}(\mu')(\mu-\mu')>\delta|Df(\mu')|.\]
			This completes the proof of proposition.
		\end{itemize}
		
	\end{proof}
	Using previous proposition, we have the following result originated from \cite{Guan14,Szekelyhidi18,CHZ21}. 
	It will play an important role in the proof of Theorem \ref{Thm4.1}.

	\begin{proposition}\label{prop subsolution}
		Let $\sigma\in[\inf_{M}h,\sup_{M}h]$ and $A$ be a Hermitian matrix with eigenvalues $\mu(A)\in \partial \Gamma^{\sigma}$.
		\begin{enumerate}\setlength{\itemsep}{1mm}
			\item There exists a constant $\tau$ depending 
			on $f$, $\Gamma$ and $\sigma$ such that
			\begin{equation}\label{lower bound of F}
				\mathcal{G}(A) = \sum_{i}G^{i\ol{i}}(A) > \tau.
			\end{equation}
			\item For $\delta, R>0$, there exists $\theta>0$ depending only on $f$, $\Gamma$, $h$, $\delta$, $R$ such that the following holds. If $B$ is a Hermitian matrix satisfying
			\begin{equation*}
				(\mu(B)-2\delta{\bf 1}+\Gamma_{n})\cap\partial\Gamma^{\sigma}\subset B_{R}(0),
			\end{equation*}
			then we  have either
			\begin{equation}\label{property 1 of sub}
				\sum_{p,q}G^{p\ol{q}}(A)[B_{p\ol{q}}-A_{p\ol{q}}]>\theta\sum_{p}G^{p\ol{p}}(A)
			\end{equation}
			or
			\begin{equation}\label{property 2 of sub}
				G^{i\ol{i}}(A)>\theta\sum_{p}G^{p\ol{p}}(A), \quad \forall \ i=1,2,\cdots, n.
			\end{equation}
		\end{enumerate}
	\end{proposition}
	
	\begin{proof}
		For (1), choosing $\sigma'$ with $\sup_{\Gamma} f>\sigma'>\sigma.$	
		By assumption (iii) and concavity, there exists a large constant $N$ such that 
		\[\sigma'<f(\mu(A)+N \textbf{1})\leq f(\mu(A)) +N\sum_{i} f_{i}(\mu(A)).\]
		It follows $\mathcal{G}(A)\geq \frac{1}{N}(\sigma'-\sigma)$ which implies (1).
		
		\medskip
		For (2), we divide into two possibilities:
		\begin{itemize}
		\itemsep0.2em
		    \item $|\mu(A)|\geq R$. We note that the proof of \cite[Proposition 6]{Szekelyhidi18} only needs assumption (i) and (ii). Then the conclusion follows.
		    \vspace{0.1cm}
		    \item $|\mu(A)|<R$. Using the argument of \cite[Proposition 2.1]{CHZ21}, we complete the proof.
		\end{itemize}
	\end{proof}


	\subsection{The dHYM equation}
	Let $\Gamma=\{\mu\in\mathbb{R}^{n}\ : \ 0<\phi(\mu)<\pi\}$ 
	and let $\phi$ be the function defined in \eqref{DHYM}. We consider the dHYM equation
	\begin{equation}\label{DHYM2}
		f(\mu(u))=\cot\phi(\mu(u))=\cot h, \quad \mu(u)\in \Gamma.
	\end{equation}
	For any $\sigma\in \mathbb{R}$, we have $\Gamma^{\sigma}=\{\mu\in\mathbb{R}^{n}\ : \ 0<\phi(\mu)<\mathrm{arccot}\ \sigma\}.$
	
		Now we prove the dHYM equation satisfying the structural condition (iii).

		\begin{proposition}\label{prop subsolution dhym}
			Let $f(\mu)=\cot \phi(\mu)$. For any $\sigma, \sigma'\in \mathbb{R}$ with $\sigma< \sigma'$,  there exists a positive constant $N$, depending only on  $\sigma$ and $\sigma'$, such that 
			\begin{equation}\label{condition iii}
				\Gamma^{\sigma}+N \textbf{1}\subset \Gamma^{\sigma'}.
			\end{equation}
	\end{proposition}
	\begin{proof}
		We fix an arbitrary $\mu\in \Gamma^{\sigma}$. By \cite[Lemma 2.1]{CPW}, there exists a constant $N'$ such that $\mu+N'\textbf{1}\in \Gamma_{n}$. It is straightforward that there exists a constant $N''$ such that $f(N''\textbf{1})>\sigma'$.
		Then we have $f(\mu+(N'+N'') \textbf{1})>f(N''\textbf{1})>\sigma'$. This implies 
		\eqref{condition iii} by letting $N=N'+N''$.
		%
	\end{proof}

	\section{A priori estimates}
	\subsection{Zero order estimate}\label{zero order estimate}

	%
	%

	\begin{proposition}\label{Prop3.2}
Let u (resp. $\underline{u}$) be a smooth solution (resp. $\mathcal{C}$-subsolution) of \eqref{nonlinear equation} with $\sup_{M}(u-\underline{u})=0$.  Then there exists a constant $C$ depending on $\underline{u}$, $h$, $\|\omega\|_{C^{0}}$, $f$, $\Gamma$ and $(M,\chi,J)$ such that
		\begin{equation*}
			\|u\|_{L^{\infty}}\leq C.
		\end{equation*}
	\end{proposition}

	\begin{proof}
		Without loss of generality, we may assume that $\underline{u}=0$.
		Thanks to \cite[(44)]{Szekelyhidi18}, we have $\mathrm{tr}_{\chi}\omega_{u}>0$ and hence
		\begin{equation}\label{Delta}
			\Delta u=\Delta_{\chi} u+\chi^{i\bar{j}}Z_{i\bar j}(\partial u)=\text{tr}_{\chi}\omega_{u}-\text{tr}_{\chi}\omega\geq -C,
		\end{equation}
		where $\Delta_{\chi}$ denotes the canonical Laplacian operator of $\chi$.  Following a similar argument of \cite[Proposition 2.3]{CTW19}, then there exists a uniform constant $C$ such that
		\begin{equation}\label{L1}
			\int_{M}(-u)\chi^{n}\leq C\,.
		\end{equation}.

		Now it suffices to establish the lower bound of the infimum $I=\inf_{M}u$.  We can adopt the arguments in \cite{CHZ21}. We remark that the only difference here is the presence of the term $Z(\partial u)$ in the definition of $H(u)$. However, this term is linear in $\partial u$, which can be controlled (by $\varepsilon$) on the contact set $P$ in \cite{CHZ21}.
	\end{proof}

	\subsection{Second order estimate}\label{second order estimate}
	In this subsection, we give the proof of Theorem \ref{main estimate}. Our first goal is the following theorem:
	\begin{theorem}\label{Thm4.1}
	Under the same  assumptions as in Proposition \ref{Prop3.2}.  Then there exists a constant $C$ depending on $\underline{u}$, $h$, $\|\omega\|_{C^{2}}$, $f$, $\Gamma$ and $(M,\chi,J)$ such that 
		\begin{equation}\label{HMW}
			\sup_{M}|\nabla^{2}u|_{\chi} \leq C(\sup_{M}|\partial u|_{\chi}^{2}+1),
		\end{equation}
		where $\nabla$ denotes the Levi-Civita connection with respect to $\chi$.
	\end{theorem}
	
	Without loss of generality, we assume $\underline{u}=0$ and $\sup_{M}u=-1$. Let $\lambda_{1}\geq\lambda_{2}\geq\cdots\geq\lambda_{2n}$ be the eigenvalues of $\nabla^{2}u$ with respect to $\chi$. For notational convenience, we write
	$|\cdot| = |\cdot|_{\chi}.$
	
Let us define
	\begin{equation}
		K=\sup_{M}|\partial u|^{2}+1, \quad  N=\sup_{M} |\nabla^{2} u|+1, \quad \rho=\nabla^{2}u+N\chi.
	\end{equation}
	On an open set $\Omega=\{\lambda_{1}>0\}\subset M$, we consider
	\begin{equation*}
		Q = \log \lambda_{1}+\varphi(|\rho|^{2})+\psi(|\partial u|^{2})+e^{-Au}
	\end{equation*}
	for a large constant $A$ to be chosen later, where
	\[
	\varphi(s) = -\frac{1}{4}\log(5N^{2}-s), \qquad
	\psi(s) = -\frac{1}{4}\log(2K-s)
	.\]
	By a directly calculation we see that
	\begin{equation}\label{xieta}
		\begin{split}
			\varphi''=4(\varphi')^{2},& \qquad \frac{1}{20N^{2}} \leq \varphi' \leq \frac{1}{16N^{2}},\\[1mm]
			\psi''=4(\psi')^{2}, & \qquad  \frac{1}{8K} \leq \psi' \leq \frac{1}{4K}.
		\end{split}
	\end{equation}
	We may assume $\Omega\neq \emptyset$, otherwise we are done. Since $Q(z)\ri -\infty$ as $z$ approaches to the boundary of  $\Omega$, we further assume $Q$ achieves its maximum at a point $x_{0}\in \Omega$. It is easy to show that (see  \cite{CHZ21} ) 
	\begin{itemize}
	\itemsep0.2em
	    \item[a)] $|\nabla^{2}u| \leq C\lambda_{1}+C\sup_{M}|\de u|+C.$
	    \vspace{0.1cm}
	    \item[b)] $\sup_{M} |\nabla^{2} u|+1 = N \leq C_{A}\lambda_{1}(x_{0}).$
	\end{itemize}
Here $C_{A}$ is a constant depending also on $A$. 
	Therefore, to prove Theorem \ref{Thm4.1}, it suffices to show
	\begin{equation}\label{goal}
		\lambda_{1}(x_{0}) \leq CK.
	\end{equation}

	Near $x_{0}$, there exists a local unitary frame $\{e_{i}\}_{i=1}^{n}$ with respect to $\chi$ such that
	\begin{equation}\label{5..6}
		\chi_{i\overline{j}}=\delta_{ij}, \quad \tilde{g}_{i\overline{j}}=\delta_{ij}\tilde{g}_{i\overline{i}}, \quad \tilde{g}_{1\overline{1}}\geq\tilde{g}_{2\overline{2}}\geq\cdots\geq\tilde{g}_{n\overline{n}} \quad \textrm{at}\ x_0.
	\end{equation}
	Here $\ti{g}_{i\ov{j}}$ is defined by $\omega_{u}=\sqrt{-1}\ti{g}_{i\ov{j}}\theta^{i}\wedge\ov{\theta}^{j}$  and $\{\theta^{i}\}_{i=1}^{n}$ denotes the dual coframe of $\{e_{i}\}_{i=1}^{n}$. It then follows that at $x_{0}$,
	\[
	G^{1\ov{1}} \leq G^{2\ov{2}} \leq \cdots \leq G^{n\ov{n}}.
	\]
	We remark that $\chi$ and $J$ are compatible implies there exists a coordinate system 
	$(U,\{x^{\alpha}\}_{\alpha=1}^{2n})$ 
	in a neighborhood of $x_{0}$ such that at $x_{0}$,
	\begin{itemize}
	\itemsep0.2em	
	    \item[a)] $e_{i}=\frac{1}{\sqrt{2}}\left(\de_{2i-1}-\sqrt{-1}\de_{2i}\right)$ for $i=1,2,\cdots,n$.
	    \vspace{0.1cm}
	    \item[b)] $\de_{\gamma} \chi_{\alpha\beta}=0$ for $\alpha,\beta,\gamma=1,2,\cdots,2n$.
	\end{itemize}
	Here $\chi_{\alpha\beta}=\chi(\partial_{\alpha}, \partial_{\beta})$ and $\de_{\alpha}=\frac{\de}{\de x^{\alpha}}$. Let us define $u_{\alpha\beta}=(\nabla^{2}u)(\de_{\alpha},\de_{\beta})$ and 
	$
	\Phi_{\beta}^{\alpha} =\sum_{\gamma=1}^{2n} \chi^{\alpha\gamma}u_{\gamma\beta},
	$
	where $(\chi^{\alpha\gamma})=(\chi_{\alpha\gamma})^{-1}$ denotes the inverse matrix of $(\chi_{\alpha\gamma})$. Clearly, $\lambda_{\alpha}$ are eigenvalues of $\Phi$. Let $V_{1},V_{2},\cdots,V_{2n}$ be the eigenvectors for $\Phi$ at $x_{0}$, corresponding to eigenvalues $\lambda_{1},\lambda_{2},\cdots,\lambda_{2n}$ respectively. Define $V_{\alpha}^{\beta}$ by $V_{\alpha}=V_{\alpha}^{\beta}\de_{\beta}$ at $x_{0}$, and extend $V_{\alpha}$ to be vector fields near $x_{0}$ by taking the components to be constants. Using a viscosity argument adapted in \cite{CTW19}, we may assume that $\lambda_{1}$ is smooth and $\lambda_{1}>\lambda_{2}$ at $x_{0}$.

	\medskip
	Applying the maximum principle  at $x_{0}$, we see that
	\begin{equation}\label{5..10}
		\frac{(\lambda_{1})_{i}}{\lambda_{1}}=-\varphi'(|\rho|^{2})_{i}-\psi'(|\de u|^{2})_{i}+Ae^{-Au}u_{i}
	\end{equation}
	for each $1\leq i\leq n$, and
	\begin{equation}\label{L Q}
		\begin{split}
			0  \geq L(Q) 
			=& \frac{L(\lambda_{1})}{\lambda_{1}}- G^{i\bar{i}}\frac{|(\lambda_{1})_{i}|^{2}}{\lambda_{1}^{2}}
			+\varphi' L(|\rho|^{2})+\varphi''G^{i\bar{i}}|(|\rho|^{2})_{i}|^{2}
			\\[1mm]
			+&\psi'L(|\partial u|^{2})+\psi'' G^{i\bar{i}}|(|\de u|^{2})_{i}|^{2} -Ae^{-Au}L(u)+A^{2}e^{-Au} G^{i\bar{i}}|u_{i}|^{2}.
		\end{split}
	\end{equation}
	In the sequel, we shall make the following conventions:
	\begin{enumerate}
		\item[(i)] all the calculations are done at $x_{0}$,
		\item[(ii)] we will use the Einstein summation,
		\item[(iii)] we usually use $C$ to denote a constant depending on $\|u\|_{C^{0}}$, $h$, $\omega$, $\Gamma$, $(M,\chi,J)$, and $C_{A}$ to denote a constant further depending  on $A$,
		\item[(iv)] we always assume without loss of generality, that  $\lambda_{1}\geq CK$ for some $C$, or $\lambda_{1}\geq C_{A}K$ for some $C_{A}$,
		\item[(v)] we use subscripts $i$ and $\bar{j}$ to denote the partial derivatives $e_{i}$ and $\bar{e}_{j}$.
	\end{enumerate}
	

	\subsubsection{Lower bound for $L(Q)$}
	\begin{proposition}\label{lower bound of L Q}
		For $\ve\in (0,\frac{1}{3}]$, at $x_{0}$, we have
		\begin{equation}\label{4.7"}
			\begin{split}
				0  \geq L(Q) & \geq
				(2-\ve)\sum_{\alpha>1}\frac{G^{i\ov{i}}|u_{V_{1}V_{\alpha}i}|^2}{\lambda_{1}(\lambda_{1}-\lambda_{\alpha})}-\frac{1}{\lambda_{1}} G^{i\bar{k},j\bar{l}}V_{1}(\tilde{g}_{i\bar{k}})V_{1}(\tilde{g}_{j\bar{l}})\\ &
				+\sum_{\alpha,\beta}\frac{G^{i\ov{i}}|e_{i}(u_{\alpha\beta})|^{2}}
				{C_{A}\lambda_{1}^{2}}
				-(1+\ve)G^{i\bar{i}}\frac{|(\lambda_{1})_{i}|^{2}}{\lambda_{1}^{2}} +\varphi''G^{i\bar{i}}|(|\rho|^{2})_{i}|^{2}\\ &
				+ \frac{3\psi'}{4}\sum_{j} G^{i\bar{i}}(|e_{i}e_{j}u|^{2}+|e_{i}\bar{e}_{j}u|^{2})
				+\psi'' G^{i\bar{i}}|(|\de u|^{2})_{i}|^{2}\\&
				-Ae^{-Au}L(u)+A^{2}e^{-Au} G^{i\bar{i}}|u_{i}|^{2}-\frac{C}{\ve}\mathcal{G}.
			\end{split}
		\end{equation}
	\end{proposition}
	\medskip
	We remark that the fourth term is the bad term that we need to control. Since $F$ is both concave and elliptic, then the first, second and third term are nonnegative,  which play an imporant role in our proof of Theorem \ref{Thm4.1}. To prove Proposition \ref{lower bound of L Q}, we shall estimate the lower bounds of $L(\lambda_{1})$, $L(|\rho|^{2})$ and $L(|\de u|^{2})$, respectively.

	\medskip 
	First, we give the  lower bound of $L(\lambda_{1})$.
	\begin{lemma}\label{lower bound of L lambda1}
		For each $\ve\in(0,\frac{1}{3}]$, at $x_{0}$, we have
		\[
		L(\lambda_{1}) \geq
		(2-\ve)\sum_{\alpha>1}\frac{G^{i\bar{i}}|e_{i}(u_{V_{\alpha}V_{1}})|^{2}}{\lambda_{1}-\lambda_{\alpha}}
		-G^{i\bar{k},j\bar{l}}V_{1}(\tilde{g}_{i\bar{k}})V_{1}(\tilde{g}_{j\bar{l}})
		-\ve G^{i\bar{i}}\frac{|(\lambda_{1})_{i}|^{2}}{\lambda_{1}}
		-\frac{C}{\ve}\lambda_{1}\mathcal{G}.
		\]
	\end{lemma}
	
	\begin{proof}
		The following formulas are well-known (see e.g. \cite{CTW19,Spruck05,Szekelyhidi18}):
		\begin{equation*}
			\begin{split}
				\frac{\partial \lambda_{1}}{\partial \Phi^{\alpha}_{\beta} }
				= {} & V_{1}^{\alpha}V_{1}^{\beta}, \\
				\frac{\partial^{2} \lambda_{1}}{\partial \Phi^{\alpha}_{\beta}\partial \Phi^{\gamma}_{\delta}}
				= {} & \sum_{\mu>1}\frac{V_{1}^{\alpha}V_{\mu}^{\beta}V_{\mu}^{\gamma}V_{1}^{\delta} +V_{\mu}^{\alpha}V_{1}^{\beta}V_{1}^{\gamma}V_{\mu}^{\delta}}{\lambda_{1}-\lambda_{\mu}}.
			\end{split}
		\end{equation*}
		Then we compute
		\begin{equation}\label{lower bound of L lambda1 eqn 1}
			\begin{split}
				L(\lambda_{1})
				= {} & G^{i\bar{i}}\frac{\partial^{2} \lambda_{1}}{\partial \Phi^{\alpha}_{\beta}\partial \Phi^{\gamma}_{\delta}}e_{i}(\Phi^{\gamma}_{\delta})\bar{e}_{i}(\Phi^{\alpha}_{\beta})
				+G^{i\bar{i}}\frac{\partial \lambda_{1}}{\partial\Phi^{\alpha}_{\beta}}(e_{i}\bar{e}_{i}
				-[e_{i},\bar{e}_{i}]^{(0,1)})(\Phi^{\alpha}_{\beta})\\
				& +G^{i\bar{i}}\frac{\partial \lambda_{1}}{\partial\Phi^{\alpha}_{\beta}}\big(e_{p}(\Phi^{\alpha}_{\beta})Z^{p}_{i\bar i}+\bar{e}_{p}(\Phi^{\alpha}_{\beta})\overline{Z^{p}_{i\bar i}}\big)\\
				= {} & G^{i\bar{i}}\frac{\partial^{2} \lambda_{1}}{\partial \Phi^{\alpha}_{\beta}\partial \Phi^{\gamma}_{\delta}}e_{i}(u_{\gamma\delta})\bar{e}_{i}(u_{\alpha\beta})+G^{i\bar{i}}\frac{\partial \lambda_{1}}{\partial\Phi^{\alpha}_{\beta}}(e_{i}\bar{e}_{i}-[e_{i},\bar{e}_{i}]^{(0,1)})(u_{\alpha\beta})\\
				&+G^{i\bar{i}}\frac{\partial \lambda_{1}}{\partial\Phi^{\alpha}_{\beta}}u_{\gamma\beta}e_{i}\bar{e}_{i}(\chi^{\alpha\gamma})+G^{i\bar{i}}\frac{\partial \lambda_{1}}{\partial\Phi^{\alpha}_{\beta}}\big(e_{p}(\Phi^{\alpha}_{\beta})Z^{p}_{i\bar i}+\bar{e}_{p}(\Phi^{\alpha}_{\beta})\overline{Z^{p}_{i\bar i}}\big)\\
				\geq {} & 2\sum_{\alpha>1}G^{i\bar{i}}\frac{|e_{i}(u_{V_{\alpha}V_{1}})|^{2}}{\lambda_{1}-\lambda_{\alpha}}
				+G^{i\bar{i}}(e_{i}\bar{e}_{i}-[e_{i},\bar{e}_{i}]^{(0,1)})(u_{V_{1}V_{1}})\\
				&+G^{i\bar{i}}\big(e_{p}(u_{V_{1}V_{1}})Z^{p}_{i\bar i}+\bar{e}_{p}(u_{V_{1}V_{1}})\overline{{Z}^{p}_{i\bar i}}\big)-C\lambda_{1}\mathcal{G}.\\
			\end{split}
		\end{equation}
		\begin{claim}\label{claim0}
			At $x_{0}$, we have
			\begin{equation*}
				G^{i\bar{i}}\big(e_{p}(u_{V_{1}V_{1}})Z^{p}_{i\bar i}+\bar{e}_{p}(u_{V_{1}V_{1}})\overline{{Z}^{p}_{i\bar i}}\big)\geq G^{i\bar{i}}V_{1}V_{1}\big(u_{p}Z^{p}_{i\bar i}+u_{\bar{p}}\overline{Z^{p}_{i\bar i}}\big)-C\lambda_{1}\mathcal{G}.
			\end{equation*}	
		\end{claim}
		\begin{proof}
			By a direct calculation,
			\begin{equation*}
				\begin{split}
					G^{i\bar{i}}e_{p}(u_{V_{1}V_{1}})Z^{p}_{i\bar i}
					= {} & G^{i\bar{i}}e_{p}
					(V_{1}V_{1}u-(\nabla_{V_{1}}V_{1})u)Z^{p}_{i\bar i} \\
					= {} &
					G^{i\bar{i}}e_{p}V_{1}V_{1}(u)\cdot Z^{p}_{i\bar i}-O(\lambda_{1})\mathcal{G}\\
					= {} &
					G^{i\bar{i}}V_{1}V_{1}e_{p}(u)\cdot Z^{p}_{i\bar i}-O(\lambda_{1})\mathcal{G}\\
					= {} &
					G^{i\bar{i}}V_{1}V_{1}\big(u_{p}Z^{p}_{i\bar i}\big)-O(\lambda_{1})\mathcal{G}.\\
				\end{split}
			\end{equation*}	
			Here and hereafter $O(\lambda_{1})$ means the terms those can be controlled by $C\lambda_{1}$.
			Similarly, we also obtain
			\begin{equation*}
				G^{i\bar{i}}\bar{e}_{p}(u_{V_{1}V_{1}})\overline{{Z}^{p}_{i\bar i}}= 
				G^{i\bar{i}}V_{1}V_{1}\big(u_{\bar{p}}\overline{{Z}^{p}_{i\bar i}}\big)-O(\lambda_{1})\mathcal{G}.
			\end{equation*}
			Then the claim follows.
		\end{proof}
		\begin{claim}\label{claim 1}
			At $x_{0}$, we have
			\begin{equation*}
				\begin{split}
					(\mathrm{I})=&G^{i\bar{i}}(e_{i}\bar{e}_{i}-[e_{i},\bar{e}_{i}]^{(0,1)})(u_{V_{1}V_{1}})	+G^{i\bar{i}}\big(e_{p}(u_{V_{1}V_{1}})Z^{p}_{i\bar i}+\bar{e}_{p}(u_{V_{1}V_{1}})\overline{{Z}^{p}_{i\bar i}}\big)\\
					\geq  & -G^{i\bar{k},j\bar{l}}V_{1}(\tilde{g}_{i\bar{k}})V_{1}(\tilde{g}_{j\bar{l}})
					-C\lambda_{1}\mathcal{G}-2(\mathrm{II}),
				\end{split}
			\end{equation*}
			where
			\[
			(\mathrm{II})= G^{i\bar{i}}\big\{[V_{1},\bar{e}_{i}]V_{1}e_{i}(u)+[V_{1},e_{i}]V_{1}\bar{e}_{i}(u)\big\}.
			\]
		\end{claim}
		
		\begin{proof}[Proof of Claim \ref{claim 1}]
			It is clear that
			\begin{equation}\label{claim2.1}
				\begin{split}
					 G^{i\bar{i}}(&e_{i}\bar{e}_{i}-[e_{i},\bar{e}_{i}]^{(0,1)})(u_{V_{1}V_{1}}) \\
					= {} & G^{i\bar{i}}(e_{i}\bar{e}_{i}-[e_{i},\bar{e}_{i}]^{(0,1)})
					(V_{1}V_{1}u-(\nabla_{V_{1}}V_{1})u) \\
					\geq {} &
					G^{i\bar{i}}e_{i}\bar{e}_{i}V_{1}V_{1}u-G^{i\bar{i}}e_{i}\bar{e}_{i}(\nabla_{V_{1}}V_{1})u
					-G^{i\bar{i}}[e_{i},\bar{e}_{i}]^{(0,1)}V_{1}V_{1}u-C\lambda_{1}\mathcal{G}.
				\end{split}
			\end{equation}	
			Set $W=\nabla_{V_{1}}V_{1}$. Then
			\begin{equation*}
				\begin{split}
					e_{i}\bar{e}_{i}W(u)
					= {} & e_{i}W\bar{e}_{i}(u)+e_{i}[\bar{e}_{i},W](u)\\
					= {} & We_{i}\bar{e}_{i}(u)+[e_{i},W]\bar{e}_{i}(u)+e_{i}[\bar{e}_{i},W](u)\\
					= {} & W(\tilde{g}_{i\bar{i}})+W[e_{i},\bar{e}_{i}]^{(0,1)}(u)
					+[e_{i},W]\bar{e}_{i}(u)+e_{i}[\bar{e}_{i},W](u)+O(\lambda_{1}).
				\end{split}
			\end{equation*}
			Applying $W$ to the equation \eqref{nonlinear equation},
			\begin{equation*}
				G^{i\bar{i}}W(\tilde{g}_{i\bar{i}})=W(h).
			\end{equation*}
			It follows that
			\[
			|G^{i\bar{i}}e_{i}\bar{e}_{i}W(u)|
			= |G^{i\bar{i}}e_{i}\bar{e}_{i}(\nabla_{V_{1}}V_{1})(u)|
			\leq C\lambda_{1}\mathcal{G}.
			\]
			Combining this with \eqref{claim2.1},
			\begin{equation}\label{claim 1 eqn 1}
				\begin{split}
	    G^{i\bar{i}}(e_{i}\bar{e}_{i}&-[e_{i},\bar{e}_{i}]^{(0,1)})(u_{V_{1}V_{1}}) \\
					\geq {} & G^{i\bar{i}}\big\{e_{i}\bar{e}_{i}V_{1}V_{1}(u)
					-[e_{i},\bar{e}_{i}]^{(0,1)}V_{1}V_{1}(u)\big\}-C\lambda_{1}\mathcal{G}.
				\end{split}
			\end{equation}
			By direct calculation, we see that
			\begin{equation*}
				\begin{split}
					 G^{i\bar{i}}&\big\{e_{i}\bar{e}_{i}V_{1}V_{1}(u)
					-[e_{i},\bar{e}_{i}]^{(0,1)}V_{1}V_{1}(u)\big\}\\
					= {} & G^{i\bar{i}}\big\{e_{i}V_{1}\bar{e}_{i}V_{1}(u)-e_{i}[V_{1},\bar{e}_{i}]V_{1}(u)
					-V_{1}[e_{i},\bar{e}_{i}]^{(0,1)}V_{1}(u)\big\}+O(\lambda_{1})\mathcal{G} \\
					= {} & G^{i\bar{i}}\big\{V_{1}e_{i}\bar{e}_{i}V_{1}(u)-[V_{1},e_{i}]\bar{e}_{i}V_{1}(u)-[V_{1},\bar{e}_{i}]e_{i}V_{1}(u)
					-V_{1}V_{1}[e_{i},\bar{e}_{i}]^{(0,1)}(u)\big\}+O(\lambda_{1})\mathcal{G}\\	
					= {} & G^{i\bar{i}}\big\{V_{1}e_{i}\bar{e}_{i}V_{1}(u)-V_{1}V_{1}[e_{i},\bar{e}_{i}]^{(0,1)}(u)\big\}
					+O(\lambda_{1})\mathcal{G}-(\mathrm{II})\\
					= {} & G^{i\bar{i}}\big\{V_{1}e_{i}V_{1}\bar{e}_{i}(u)-V_{1}e_{i}[V_{1},\bar{e}_{i}](u)-V_{1}V_{1}[e_{i},\bar{e}_{i}]^{(0,1)}(u)\big\}
					+O(\lambda_{1})\mathcal{G}-(\mathrm{II})\\
					= {} & G^{i\bar{i}}\big\{V_{1}V_{1}(e_{i}\bar{e}_{i}-[e_{i},\bar{e}_{i}]^{(0,1)})(u)-V_{1}[V_{1},e_{i}]\bar{e}_{i}(u)-V_{1}e_{i}[V_{1},\bar{e}_{i}](u)\big\}
					+O(\lambda_{1})\mathcal{G}-(\mathrm{II})\\
					= {} & G^{i\bar{i}}V_{1}V_{1}\big(e_{i}\bar{e}_{i}(u)-[e_{i},\bar{e}_{i}]^{(0,1)}(u)\big)
					+O(\lambda_{1})\mathcal{G}-2(\mathrm{II}). \\
				\end{split}
			\end{equation*}
			Substituting this with Claim \ref{claim0} into \eqref{claim 1 eqn 1}, we obtain
			\begin{equation}\label{claim 1 eqn 2}
				(\mathrm{I})\geq G^{i\bar{i}}V_{1}V_{1}(\ti{g}_{i\ov{i}})+O(\lambda_{1})\mathcal{G}-2(\mathrm{II}).
			\end{equation}
			To deal with the first term, we apply $V_{1}V_{1}$ to the equation \eqref{nonlinear equation} and obtain
			\begin{equation}\label{claim 1 eqn 3}
				G^{i\bar{i}}V_{1}V_{1}(\tilde{g}_{i\bar{i}})
				= -G^{i\bar{k},j\bar{l}}V_{1}(\tilde{g}_{i\bar{k}})V_{1}(\tilde{g}_{j\bar{l}})+V_{1}V_{1}(h).
			\end{equation}
			Then Claim \ref{claim 1} follows from \eqref{claim 1 eqn 2} and \eqref{claim 1 eqn 3}.
		\end{proof}
		
\medskip
Using the similar argument of \cite[Claim 2]{CHZ21}, for each $\ve\in(0,\frac{1}{3}]$, we deduce
\begin{equation}\label{lower bound of II}
  2(\mathrm{II})\leq \ve\frac{G^{i\bar{i}}|(\lambda_{1})_{i}|^{2}}{\lambda_{1}}
			+\ve\sum_{\alpha>1}\frac{G^{i\bar{i}}|e_{i}(u_{V_{\alpha}V_{1}})|^{2}}{\lambda_{1}-\lambda_{\alpha}}
			+\frac{C}{\ve}\lambda_{1}\mathcal{G}.  
\end{equation}
Combining \eqref{lower bound of L lambda1 eqn 1}, \eqref{lower bound of II} and Claim 2, we obtain Lemma \ref{lower bound of L lambda1}.
	\end{proof}
	
	\medskip
	Second, we estimate the lower bound of $L(|\rho|^{2})$.
	\begin{lemma}\label{lower bound of L rho}
		For each $\ve\in(0,\frac{1}{3}]$, at $x_{0}$, we have
		\begin{equation*}
			\begin{split}
				L(|\rho|^{2})\geq (2-\varepsilon)\sum_{\alpha,\beta}G^{i\bar{i}}|e_{i}(u_{\alpha\beta})|^{2}-\frac{C}{\varepsilon}N^{2}\mathcal F.\\
			\end{split}
		\end{equation*}
	\end{lemma}
	
	\begin{proof}
		We remark that the linear gradient terms in $L$ can be absorbed by $N^{2}\mathcal F$. Thus the proof is similar to \cite{CHZ21}.
	\end{proof}

	\medskip
	Finally, we give the lower bound of $L(|\de u|^{2})$.
	
	\begin{lemma}\label{lower bound of L de u}
		At $x_{0}$, we have
		\begin{equation}\label{3.7}
			L(|\partial u|^{2}) \geq \frac{3}{4} \sum_{j}G^{i\bar{i}}(|e_{i}e_{j}u|^{2}+|e_{i}\bar{e}_{j}u|^{2})-CK\mathcal{G}.
		\end{equation}
	\end{lemma}
	
	\begin{proof}
		By a direct calculation, we deduce
		\[
		\begin{split}
			L(|\partial u|^{2})=&G^{i\bar{i}}\Big({e_{i}e_{\bar{i}}}(|\partial u|^{2})-[e_{i},\bar{e}_{i}]^{(0,1)}(|\partial u|^{2})+e_{p}(|\partial u|^{2})Z^{p}_{i\bar i}+\bar{e}_{p}(|\partial u|^{2})\overline{Z^{p}_{i\bar i}}\Big)\\
			=&  I_{1}+I_{2}+I_{3},
		\end{split}
		\]	
		where
		\[
		\begin{split}
			I_{1}
			= {} & G^{i\bar{i}}\Big(e_{i}\bar{e}_{i}e_{j}u
			-[e_{i},\bar{e}_{i}]^{(0,1)}e_{j}u+e_{p}e_{j}(u)Z^{p}_{i\bar i}+\bar{e}_{p}e_{j}(u)\overline{Z^{p}_{i\bar i}}\Big)\bar{e}_{j}u, \\
			I_{2}
			= {} & G^{i\bar{i}}\Big(e_{i}\bar{e}_{i}\bar{e}_{j}u
			-[e_{i},\bar{e}_{i}]^{(0,1)}\bar{e}_{j}u+e_{p}\bar{e}_{j}(u)Z^{p}_{i\bar i}+\bar{e}_{p}\bar{e}_{j}(u)\overline{Z^{p}_{i\bar i}}\Big)e_{j}u,\\
			I_{3}
			= {} & G^{i\bar{i}}(|e_{i}e_{j}u|^{2}
			+|e_{i}\bar{e}_{j}u|^{2}).
		\end{split}
		\]
		Applying $e_{j}$ to the equation \eqref{nonlinear equation},
		\[
		G^{i\bar{i}}e_{j}\big(e_{i}\bar{e}_{i}u-[e_{i},\bar{e}_{i}]^{(0,1)}u+e_{p}(u)Z^{p}_{i\bar i}+\bar{e}_{p}(u)\overline{Z^{p}_{i\bar i}}\big)=h_{j}.
		\]
		Note that
		\[
		\begin{split}
			 G^{i\bar{i}}\Big(&e_{i}\bar{e}_{i}e_{j}u
			-[e_{i},\bar{e}_{i}]^{(0,1)}e_{j}u+e_{p}e_{j}(u)Z^{p}_{i\bar i}+\bar{e}_{p}e_{j}(u)\overline{Z^{p}_{i\bar i}}\Big) \\
			= {} &  G^{i\bar{i}}(e_{j}e_{i}\bar{e}_{i}u+e_{i}[\bar{e}_{i},e_{j}]u
			+[e_{i},e_{j}]\bar{e}_{i}u-[e_{i},\bar{e}_{i}]^{(0,1)}e_{j}u)\\
			&+G^{i\bar{i}}(e_{j}e_{p}(u)Z^{p}_{i\bar i}+e_{j}\bar{e}_{p}(u)\overline{Z^{p}_{i\bar i}})+O(\sqrt{K})\mathcal{G}\\
			= {} &  G^{i\bar{i}}(e_{j}e_{i}\bar{e}_{i}u+e_{i}[\bar{e}_{i},e_{j}]u
			+[e_{i},e_{j}]\bar{e}_{i}u-[e_{i},\bar{e}_{i}]^{(0,1)}e_{j}u)\\
			&+G^{i\bar{i}}e_{j}(e_{p}(u)Z^{p}_{i\bar i}+\bar{e}_{p}(u)\overline{Z^{p}_{i\bar i}})+O(\sqrt{K})\mathcal{G}\\
			= {} & h_{j}+G^{i\bar{i}}e_{j}[e_{i},\bar{e}_{i}]^{(0,1)}u
			+G^{i\bar{i}}(e_{i}[\bar{e}_{i},e_{j}]u
			+[e_{i},e_{j}]\bar{e}_{i}u-[e_{i},\bar{e}_{i}]^{(0,1)}e_{j}u)+O(\sqrt{K})\mathcal{G}\\
			= {} & h_{j}+G^{i\bar{i}}\big\{e_{i}[\bar{e}_{i},e_{j}]u
			+\bar{e}_{i}[e_{i},e_{j}]u+[[e_{i},e_{j}],\bar{e}_{i}]u
			-[[e_{i},\bar{e}_{i}]^{(0,1)},e_{j}]u\big\}+O(\sqrt{K})\mathcal{G},
		\end{split}
		\]
		where $O(\sqrt{K})$ means the terms those can be controlled by $C\sqrt{K}$.
		Similarly,
		\[
		\begin{split}
	 G^{i\bar{i}}\Big(&e_{i}\bar{e}_{i}\bar{e}_{j}u
			-[e_{i},\bar{e}_{i}]^{(0,1)}\bar{e}_{j}u+e_{p}\bar{e}_{j}(u)Z^{p}_{i\bar i}+\bar{e}_{p}\bar{e}_{j}(u)\overline{Z^{p}_{i\bar i}}\Big) \\
			= {} & h_{\ov{j}}+G^{i\bar{i}}\big\{e_{i}[\bar{e}_{i},\ov{e}_{j}]u
			+\bar{e}_{i}[e_{i},\ov{e}_{j}]u+[[e_{i},\ov{e}_{j}],\bar{e}_{i}]u
			-[[e_{i},\bar{e}_{i}]^{(0,1)},\ov{e}_{j}]u\big\}+O(\sqrt{K})\mathcal{G}.
		\end{split}
		\]
		By the Cauchy-Schwarz inequality,
		\begin{equation}\label{4.9}
			\begin{split}
				I_{1}+I_{2} 
				\geq {} & 2\textrm{Re}\big(\sum_{j}h_{j}u_{\bar{j}}\big)
				-C|\partial u|\sum_{j}G^{i\bar{i}}(|e_{i}e_{j}u|+|e_{i}\bar{e}_{j}u|)-CK\mathcal{G}\\
				\geq {} & -C|\de u|
				-\frac{1}{4}\sum_{j}G^{i\bar{i}}(|e_{i}e_{j}u|^{2}+|e_{i}\bar{e}_{j}u|^{2})-CK\mathcal{G}.
			\end{split}
		\end{equation}
		Then we have
		\[
		\begin{split}
			L(|\partial u|^{2}) = {} & I_{1}+I_{2}+I_{3} 
			\geq {}  \frac{3}{4}\sum_{j}G^{i\bar{i}}(|e_{i}e_{j}u|^{2}+|e_{i}\bar{e}_{j}u|^{2})-CK\mathcal{G}.
		\end{split}
		\]
		This proves the lemma.
	\end{proof}
	
	We will use the above computations to prove Proposition \ref{lower bound of L Q}.
	
	\begin{proof}[Proof of Proposition \ref{lower bound of L Q}]
		Combining \eqref{L Q} and Lemmas \ref{lower bound of L lambda1}--\ref{lower bound of L de u}, we obtain
		\begin{equation*}
			\begin{split}
				0 \geq {} & (2-\ve)
				\sum_{\alpha>1}\frac{G^{i\bar{i}}|e_{i}(u_{V_{\alpha}V_{1}})|^{2}}{\lambda_{1}(\lambda_{1}-\lambda_{\alpha})}
				-\frac{1}{\lambda_{1}}G^{i\bar{k},j\bar{l}}V_{1}(\tilde{g}_{i\bar{k}})V_{1}(\tilde{g}_{j\bar{l}})\\
				&+(2-\varepsilon)\varphi'\sum_{\alpha,\beta}G^{i\bar{i}}|e_{i}(u_{\alpha\beta})|^{2}
				-(1+\ve)G^{i\bar{i}}\frac{|(\lambda_{1})_{i}|^{2}}{\lambda_{1}^{2}} +\varphi''G^{i\bar{i}}|(|\rho|^{2})_{i}|^{2}\\
				& +\frac{3\psi'}{4}\sum_{j} G^{i\bar{i}}(|e_{i}e_{j}u|^{2}+|e_{i}\bar{e}_{j}u|^{2})
				+\psi'' G^{i\bar{i}}|(|\de u|^{2})_{i}|^{2}\\
				& -Ae^{-Au}L(u)+A^{2}e^{-Au} G^{i\bar{i}}|u_{i}|^{2}
				-\frac{C}{\ve}(1+\varphi'N^{2}+\psi'K)\mathcal{G}.
			\end{split}
		\end{equation*}
		It suffices to deal with the third and last term. For the third term, using \eqref{xieta} and the fact $N\leq C_A \lambda_1$,
		\[
		(2-\varepsilon)\varphi'\sum_{\alpha,\beta}G^{i\bar{i}}|e_{i}(u_{\alpha\beta})|^{2}
		\geq \sum_{\alpha,\beta}\frac{G^{i\ov{i}}|e_{i}(u_{\alpha\beta})|^{2}}{20N^{2}}
		\geq \sum_{\alpha,\beta}\frac{G^{i\ov{i}}|e_{i}(u_{\alpha\beta})|^{2}}{C_{A}\lambda_{1}^{2}}.
		\]
		For the last term, using \eqref{xieta} again we infer that
		\[
		-\frac{C}{\ve}(1+\varphi'N^{2}+\psi'K)\mathcal{G}
		\geq -\frac{C}{\ve}\mathcal{G}.
		\]
		Combining the above inequalities, we conclude Proposition \ref{lower bound of L Q}.
	\end{proof}

	\subsubsection{Proof of Theorem \ref{Thm4.1}}
	First, we define the index set
	\begin{equation*}
		J=\Big\{ 1\leq j\leq n \ :  \ \frac{\psi'}{2}\sum_{i} (|e_{i}e_{j}u|^{2}+|e_{i}\bar{e}_{j}u|^{2})\geq A^{5n}e^{-5nu}K \ \,  \text{at $x_{0}$}\Big\}.
	\end{equation*}
	If $J=\emptyset$, then  Theorem \ref{Thm4.1} follows. So we assume $J\neq\emptyset$ and let $j_0$ be the maximal element of $J$. If $j_{0}<n$, we denote
	\begin{equation}\label{def of S}
		S=\Big\{j_{0}\leq i\leq n-1 \ : \
		G^{i\bar{i}} \leq A^{-2}e^{2Au}G^{i+1\overline{i+1}} \ \,  \text{at $x_{0}$}\Big\}.
	\end{equation}
	According to the index sets $J$ and $S$, the proof of Theorem \ref{Thm4.1} can be divided into three cases:
	\begin{itemize}
	\itemsep0.2em
	    \item[Case 1.] $j_{0}=n$.
	    \vspace{0.1cm}
	    \item[Case 2.] $j_{0}<n$ and $S=\emptyset$.
	    \vspace{0.1cm}
	    \item[Case 3.] $j_{0}<n$ and $S\neq\emptyset$.
	\end{itemize}
	
	\medskip
	For Case 1 and Case 2, the proof in \cite{CHZ21} is still valid in our setting, we shall omit it here.
	Now we only need to establish Case 3.

	\medskip
	Observe that $S\neq\emptyset$. Let $i_{0}$ be the minimal element of $S$ and define
	\begin{equation*}
		I=\{i_{0}+1,\cdots, n\}.
	\end{equation*}
Let us decompose the term \begin{equation}\label{definition Bi}
		\begin{split}
			 (1&+\ve)\sum_{i}G^{i\bar{i}}\frac{|(\lambda_{1})_{i}|^{2}}{\lambda_{1}^{2}}\\
			& = (1+\ve)\sum_{i\not\in I}G^{i\bar{i}}\frac{|(\lambda_{1})_{i}|^{2}}{\lambda_{1}^{2}}
			+3\ve\sum_{i\in I}G^{i\bar{i}}\frac{|(\lambda_{1})_{i}|^{2}}{\lambda_{1}^{2}}+(1-2\ve)\sum_{i\in I}G^{i\bar{i}}\frac{|(\lambda_{1})_{i}|^{2}}{\lambda_{1}^{2}}\\[2mm]
			& = B_{1}+B_{2}+B_{3}
		\end{split}
	\end{equation}
	into three terms based on $I$.

	%
	
	\begin{lemma}\label{bad terms 1 2}
		At $x_{0}$, we have
		\begin{equation*}
			\begin{split}
				B_1+B_2
				\leq {} & \frac{\psi'}{4}\sum_{j}G^{i\bar{i}} (|e_{i}e_{j}u|^{2}+|e_{i}\bar{e}_{j}u|^{2})+\varphi''G^{i\bar{i}}|(|\rho|^{2})_{i}|^{2}\\&+\psi''G^{i\bar{i}}|(|\de u|^{2})_{i}|^{2}+9\ve A^{2}e^{-2Au}G^{i\bar{i}}|u_{i}|^{2}.
			\end{split}
		\end{equation*}
	\end{lemma}
	
	\begin{proof}
		See the proof of \cite[Lemma 4.6]{CHZ21}.
	\end{proof}

	\subsubsection{Calculations of $B_3$}
We now devote to prove the following proposition.
	\begin{proposition}\label{B 3}
		Let $\ve=\frac{e^{Au(x_{0})}}{9}$. Then at $x_{0}$, we have
		\begin{equation}
			\begin{split}
				B_3\leq  &(2-\ve)
				\sum_{\alpha>1}\frac{G^{i\bar{i}}|e_{i}(u_{V_{\alpha}V_{1}})|^{2}}{\lambda_{1}(\lambda_{1}-\lambda_{\alpha})}
				-\frac{1}{\lambda_{1}}G^{i\bar{k},j\bar{l}}V_{1}(\tilde{g}_{i\bar{k}})V_{1}(\tilde{g}_{j\bar{l}})\\
				&+(2-\varepsilon)\varphi'\sum_{\alpha,\beta}G^{i\bar{i}}|e_{i}(u_{\alpha\beta})|^{2}+\frac{C}{\ve}\mathcal{G}.
			\end{split}	
		\end{equation}
	\end{proposition}
Let us define
	\begin{equation}\label{definition of W}
		W_{1} = \frac{1}{\sqrt{2}}(V_{1}-\sqrt{-1}JV_{1}) = \sum_{q}\nu_{q}e_{q}, \qquad
		JV_{1} = \sum_{\alpha>1}\mu_{\alpha}V_{\alpha},
	\end{equation}
	where we used $V_1$ is orthogonal to $JV_1$. At $x_{0}$, $V_{1}$ and $e_{q}$ are $\chi$-unitary, which implies
	\[
	\sum_{q=1}^{n}|\nu_{q}|^{2} = 1, \qquad
	\sum_{\alpha>1}\mu_{\alpha}^{2} = 1.
	\]

	\begin{lemma}\label{nu}
		At $x_{0}$, we have
		\begin{itemize}
		\itemsep0.2em
		\item[(1)] $\omega_{u}\geq -C_{A}K\chi \,,$
		\vspace{0.1cm}
		\item[(2)] $	|\nu_{i}|\leq \frac{C_{A}K}{\lambda_{1}} $ for any $i\in I\,.$
		\end{itemize}
	\end{lemma} 
	
	\begin{proof}
	Recalling the definitions of $i_0$ and $j_0$, we deduce $i_{0}+1>i_0\geq j_0$ and hence $I\cap J=\emptyset$. Therefore,
		\begin{equation}\label{nu eqn 1}
			\frac{\psi'}{4}\sum_{j} (|e_{i}e_{j}u|^{2}+|e_{i}\bar{e}_{j}u|^{2}) \leq A^{5n}e^{-5Anu}K,
			\quad \text{for each $i\in I$}.
		\end{equation}
		Furthermore,  $n\in I$  implies $e_{n}\ov{e}_{n}u\geq-C_{A}K$ and 
		\[
		\ti{g}_{n\ov{n}} = g_{n\ov{n}}+e_{n}\ov{e}_{n}u+[e_{n},\ov{e}_{n}]^{(0,1)}u+Z_{n\bar n}
		\geq e_{n}\ov{e}_{n}u-CK \geq -C_{A}K.
		\]
		Using this together with \eqref{5..6}, we conclude (1). The proof of (2) can be found in \cite[Lemma 4.8]{CHZ21}.
	\end{proof}
	Now we give the proof of Proposition \ref{B 3}.
	\begin{proof}[Proof of Proposition \ref{B 3}]
		By the definition of $W_{1}$ in \eqref{definition of W}, we see that $V_{1}=\sqrt{2}\overline{W}_{1}-\sqrt{-1}JV_{1}$. This implies
		\[
		\begin{split}
			e_{i}(u_{V_1V_1})
			= {} & -\sqrt{-1}\sum_{\alpha>1}\mu_\alpha e_{i}(u_{V_1 V_\alpha})
			+\sqrt{2}\sum_{q}\ov{\nu_q}V_{1}e_{i}\ov{e}_{q}u+O(\lambda_1) \\
			= {} & -\sqrt{-1}\sum_{\alpha>1}\mu_\alpha e_{i}(u_{V_1 V_\alpha})
			+\sqrt{2}\sum_{q\notin I}\ov{\nu_q}V_1(\tilde{g}_{i\ov{q}})
			+\sqrt{2}\sum_{q\in I}\ov{\nu_q}V_{1}e_{i}\ov{e}_{q}u+O(\lambda_1).
		\end{split}
		\]
		Using this together with Cauchy-Schwarz inequality and Lemma \ref{nu},
		\begin{equation}\label{E1}
			\begin{split}
				B_3				\leq  (1-\ve)&\sum_{i\in I}\frac{G^{i\ov{i}}}{\lambda_1^2}\left|-\sqrt{-1}\sum_{\alpha>1}\mu_\alpha e_{i}(u_{V_1 V_\alpha})+\sqrt{2}\sum_{q\notin  I}\ov{\nu_q}V_{1}(\tilde{g}_{i\ov{q}})\right|^{2} \\
				& +\frac{C_{A}}{\ve\lambda_1^2}\sum_{i\in I}\sum_{q\in I}\frac{G^{i\ov{i}}|V_{1}e_{i}\ov{e}_{q}u|^{2}}{\lambda_1^2}+\frac{C\mathcal{G}}{\ve}.\\
			\end{split}
		\end{equation}
For the second term in RHS of \eqref{E1}.		Observing that $|V_{1}e_{i}\ov{e}_{q}u| \leq C\sum_{\alpha,\beta}|e_{i}(u_{\alpha\beta})|+C\lambda_{1}$, we deduce
		\[
		\frac{C_A}{\ve\lambda_1^2}\sum_{i\in I}\sum_{q\in I}\frac{G^{i\ov{i}}|V_{1}e_{i}\ov{e}_{q}u|^{2}}{\lambda_1^2}
		\leq \frac{C_A}{\ve\lambda_1^2}\sum_{\alpha,\beta}\frac{G^{i\ov{i}}|e_{i}(u_{\alpha\beta})|^{2}}{\lambda_{1}^{2}}
		+\frac{C_A}{\ve\lambda_1^2}\mathcal{G}.
		\]
		Under the assumption $\lambda_{1}\geq\frac{C_{A}}{\ve}$, we obtain
		\begin{equation}\label{B33}
			\frac{C_A}{\ve\lambda_1^2}\sum_{i\in I}\sum_{q\in I}\frac{G^{i\ov{i}}|V_{1}e_{i}\ov{e}_{q}u|^{2}}{\lambda_1^2} \leq
			\sum_{\alpha,\beta}\frac{G^{i\ov{i}}|e_{i}(u_{\alpha\beta})|^{2}}{C_{A}\lambda_{1}^{2}}+\mathcal{G}.
		\end{equation}
		Now we deal with the first term in RHS of \eqref{E1}. For a constant $\gamma>0$ to be chosen later, we see that
		\begin{equation}\label{E2}
			\begin{split}
				\sum_{i\in I}\frac{G^{i\ov{i}}}{\lambda_1^2}&\Big|-\sqrt{-1}\sum_{\alpha>1}\mu_\alpha e_{i}(u_{V_1 V_\alpha})+\sqrt{2}\sum_{q\not\in I}\ov{\nu_q}V_{1}(\tilde{g}_{i\ov{q}})\Big|^{2}\\
				\leq {} & \Big(1+\frac{1}{\gamma}\Big)\sum_{i\in I}\frac{G^{i\ov{i}}}{\lambda_1^2}
				\Big|\sum_{\alpha>1}\mu_\alpha e_{i}(u_{V_1 V_\alpha})\Big|^{2}  +(1+\gamma)\sum_{i\in I}\frac{2G^{i\ov{i}}}{\lambda_{1}^{2}}\Big|\sum_{q\notin I}\ov{\nu_q}V_1(\tilde{g}_{i\ov{q}})\Big|^{2}.
			\end{split}
		\end{equation}
	 Using the Cauchy-Schwarz inequality again, 	for the first term,
		\begin{equation}\label{B31}
			\begin{split}
			\Big(1+\frac{1}{\gamma}\Big)\sum_{i\in I}&\frac{G^{i\ov{i}}}{\lambda_1^2}
				\Big|\sum_{\alpha>1}\mu_\alpha e_{i}(u_{V_1 V_\alpha})\Big|^{2}\\
				\leq {} & \Big(1+\frac{1}{\gamma}\Big)
				\sum_{i\in I}\frac{G^{i\ov{i}}}{\lambda_1^2}\Big(\sum_{\alpha>1}(\lambda_{1}-\lambda_{\alpha})\mu_{\alpha}^{2}\Big)
				\Big(\sum_{\alpha>1}\frac{|e_{i}(u_{V_{1}V_{\alpha}})|^{2}}{\lambda_1-\lambda_\alpha}\Big) \\
				= {} & \Big(1+\frac{1}{\gamma}\Big)\sum_{i\in I}\frac{G^{i\ov{i}}}{\lambda_1^2}\Big(\lambda_{1}-\sum_{\alpha>1}\lambda_{\alpha}\mu_{\alpha}^{2}\Big)
	           \Big(\sum_{\alpha>1}\frac{|e_{i}(u_{V_{1}V_{\alpha}})|^{2}}{\lambda_1-\lambda_\alpha}\Big)\,, \\
			\end{split}
		\end{equation}
	and	for the second term, 
		\[
		\begin{split}
				(1+\gamma)\sum_{i\in I}&\frac{2G^{i\ov{i}}}{\lambda_{1}^{2}}\Big|\sum_{q\notin I}\ov{\nu_q}V_1(\tilde{g}_{i\ov{q}})\Big|^{2} \\
			\leq {} & (1+\gamma)\sum_{i\in I}\frac{2G^{i\ov{i}}}{\lambda_{1}^{2}}
			\Big(\sum_{q\notin I}
			\frac{(\tilde{g}_{q\ov{q}}-\tilde{g}_{i\ov{i}})|\nu_{q}|^{2}}{G^{i\ov{i}}-G^{q\ov{q}}}\Big)
			\Big(\sum_{q\notin I}
			\frac{(G^{i\ov{i}}-G^{q\ov{q}})|V_{1}(\tilde{g}_{i\ov{q}})|^{2}}{\tilde{g}_{q\ov{q}}-\tilde{g}_{i\ov{i}}}\Big).
		\end{split}
		\]
		Recalling the definition of the index set $I$, when  $q\notin I$ and $i\in I$,
		\[
		G^{q\ov{q}} \leq G^{i_0\ov{i_0}} \leq A^{-2}e^{2Au}G^{i_0+1\ov{i_0+1}} \leq A^{-2}e^{2Au}G^{i\ov{i}}.
		\]
		Combining this with Lemma \ref{nu},
		\begin{equation}\label{positive constant}
			0 <
			\frac{(\tilde{g}_{q\ov{q}}-\tilde{g}_{i\ov{i}})|\nu_{q}|^{2}}{G^{i\ov{i}}-G^{q\ov{q}}}
			\leq \frac{\tilde{g}_{q\ov{q}}|\nu_{q}|^{2}-\tilde{g}_{i\ov{i}}|\nu_{q}|^{2}}{(1-A^{-2}e^{2Au})G^{i\ov{i}}}
			< \frac{\tilde{g}_{q\ov{q}}|\nu_{q}|^{2}+C_{A}K}{(1-A^{-2}e^{2Au})G^{i\ov{i}}}.
		\end{equation}
In addition, from  \eqref{second derive of F} and the concavity of $f$, we get
		\begin{equation}\label{second sum}
		   -\frac{1}{\lambda_{1}} G^{i\bar{k},j\bar{l}}V_{1}(\tilde{g}_{i\bar{k}})V_{1}(\tilde{g}_{j\bar{l}})
		\geq \frac{2}{\lambda_{1}}\sum_{i\in I}\sum_{q\notin I}\frac{(G^{i\ov{i}}-G^{q\ov{q}})|V_{1}(\tilde{g}_{i\ov{q}})|^{2}}{\tilde{g}_{q\ov{q}}-\tilde{g}_{i\ov{i}}}. 
		\end{equation}
	It follows from \eqref{positive constant} and \eqref{second sum} that
		\begin{equation}\label{E3}
			\begin{split}
					(1+\gamma)\sum_{i\in I}&\frac{2G^{i\ov{i}}}{\lambda_{1}^{2}}\Big|\sum_{q\notin I}\ov{\nu_q}V_1(\tilde{g}_{i\ov{q}})\Big|^{2} \\
				\leq {} & \frac{(1+\gamma )}{\lambda_{1}(1-A^{-2}e^{2Au})}
				\Big(\sum_{q\notin I}\tilde{g}_{q\ov{q}}|\nu_{q}|^{2}+C_{A}K\Big)\cdot \Big\{-\frac{1}{\lambda_{1}} G^{i\bar{k},j\bar{l}}V_{1}(\tilde{g}_{i\bar{k}})V_{1}(\tilde{g}_{j\bar{l}})\Big\}.
			\end{split}
		\end{equation}
		Since $\ve=\frac{e^{Au(x_{0})}}{9}$, when $A$ is large enough one have
		\begin{equation}\label{E4}
			\frac{(1-\ve)(1+\gamma)}{\lambda_{1}(1-A^{-2}e^{2Au})}
			\leq \Big(1-\frac{\ve}{2}\Big)\Big(\frac{1+\gamma}{\lambda_{1}}\Big).
		\end{equation}
		Together with \eqref{E2}, \eqref{B31}, \eqref{E3} and \eqref{E4}, we conclude
		\begin{equation}\label{E5}
			\begin{split}
					(1-\ve)\sum_{i\in I}&\frac{G^{i\ov{i}}}{\lambda_1^2}
				\Big|-\sqrt{-1}\sum_{\alpha>1}\mu_\alpha e_{i}(u_{V_1 V_\alpha})+\sqrt{2}\sum_{q\notin  I}\ov{\nu_q}V_{1}(\tilde{g}_{i\ov{q}})\Big|^{2} \\
				\leq {} & (1-\ve)\Big(1+\frac{1}{\gamma}\Big)\sum_{i\in I}\frac{G^{i\ov{i}}}{\lambda_1^2}\Big(\lambda_{1}-\sum_{\alpha>1}\lambda_{\alpha}\mu_{\alpha}^{2}\Big)
				\Big(\sum_{\alpha>1}\frac{|e_{i}(u_{V_{1}V_{\alpha}})|^{2}}{\lambda_1-\lambda_\alpha}\Big)\\
				& +\frac{(1-\ve)(1+\gamma )}{\lambda_{1}(1-A^{-2}e^{2Au})}
				\Big(\sum_{q\notin I}\tilde{g}_{q\ov{q}}|\nu_{q}|^{2}+C_{A}K\Big)\cdot \Big\{-\frac{1}{\lambda_{1}} G^{i\bar{k},j\bar{l}}V_{1}(\tilde{g}_{i\bar{k}})V_{1}(\tilde{g}_{j\bar{l}})\Big\}\\
				\leq {} & \frac{1-\ve}{(2-\ve)\lambda_{1}}
				\Big(1+\frac{1}{\gamma}\Big)\Big(\lambda_{1}-\sum_{\alpha>1}\lambda_{\alpha}\mu_{\alpha}^{2}\Big)\cdot\Big\{(2-\ve)\sum_{\alpha>1}\frac{G^{i\ov{i}}|u_{V_{1}V_{\alpha}i}|^2}{\lambda_{1}(\lambda_{1}-\lambda_{\alpha})}\Big\}\\
				& +\left(1-\frac{\ve}{2}\right)\left(\frac{1+\gamma}{\lambda_{1}}\right)
				\Big(\sum_{q\notin I}\tilde{g}_{q\ov{q}}|\nu_{q}|^{2}+C_{A}K\Big)\cdot \Big\{-\frac{1}{\lambda_{1}} G^{i\bar{k},j\bar{l}}V_{1}(\tilde{g}_{i\bar{k}})V_{1}(\tilde{g}_{j\bar{l}})\Big\}.\\
			\end{split}
		\end{equation}
		Now we prove the  following lemma:
		\begin{lemma}
			At $x_{0}$, we have
			\begin{equation}\label{bad}
				\begin{split}
						(1-\ve)\sum_{i\in I}&\frac{G^{i\ov{i}}}{\lambda_1^2}
					\Big|-\sqrt{-1}\sum_{\alpha>1}\mu_\alpha e_{i}(u_{V_1 V_\alpha})+\sqrt{2}\sum_{q\notin  I}\ov{\nu_q}V_{1}(\tilde{g}_{i\ov{q}})\Big|^{2} \\
					\leq {} & (2-\ve)\sum_{\alpha>1}\frac{G^{i\ov{i}}|u_{V_{1}V_{\alpha}i}|^2}{\lambda_{1}(\lambda_{1}-\lambda_{\alpha})}-\frac{1}{\lambda_{1}} G^{i\bar{k},j\bar{l}}V_{1}(\tilde{g}_{i\bar{k}})V_{1}(\tilde{g}_{j\bar{l}}).
				\end{split}
			\end{equation}
		\end{lemma}
		\begin{proof} In light of \eqref{E5}, it suffices to prove
		\begin{itemize}
		 \itemsep0.2em
		    \item[a)] $	\frac{1-\ve}{(2-\ve)\lambda_{1}}
		\big(1+\frac{1}{\gamma}\big)\big(\lambda_{1}-\sum_{\alpha>1}\lambda_{\alpha}\mu_{\alpha}^{2}\big)\leq 1.$
		    \vspace{0.1cm}
		    \item[b)] $\big(1-\frac{\ve}{2}\big)\big(\frac{1+\gamma}{\lambda_{1}}\big)
				\big(\sum_{q\notin I}\tilde{g}_{q\ov{q}}|\nu_{q}|^{2}+C_{A}K\big)\leq 1.$
		\end{itemize}
		
		\medskip
			We shall consider the following two cases:
			\bigskip
\begin{itemize}
    \item[Case A.] $\frac{1}{2}\big(\lambda_{1}+\sum_{\alpha>1}\lambda_{\alpha}\mu_{\alpha}^{2}\big)
			>\big(1-\frac{\ve}{2}\big)\big(\sum_{q\notin I}\tilde{g}_{q\ov{q}}|\nu_{q}|^{2}+C_{A}K\big)$. 
			
			\medskip
			It follows from \eqref{positive constant} that
			\[
			\frac{1}{2}\big(\lambda_{1}+\sum_{\alpha>1}\lambda_{\alpha}\mu_{\alpha}^{2}\big)
			>\big(1-\frac{\ve}{2}\big)\big(\sum_{q\notin I}\tilde{g}_{q\ov{q}}|\nu_{q}|^{2}+C_{A}K\big)
			\geq 0.
			\]
In this case we set $\gamma= \frac{\lambda_{1}-\sum_{\alpha>1}\lambda_{\alpha}\mu_{\alpha}^{2}}{\lambda_{1}+\sum_{\alpha>1}\lambda_{\alpha}\mu_{\alpha}^{2}}.$		Note that $\lambda_{1}>\lambda_{2}$ at $x_{0}$ and so $\gamma$ is positive. This concludes $a)$ and  $b)$.
\vspace{0.1cm}
   \item[Case B.] $\frac{1}{2}\left(\lambda_{1}+\sum_{\alpha>1}\lambda_{\alpha}\mu_{\alpha}^{2}\right)
			\leq(1-\frac{\ve}{2})\left(\sum_{q\notin I}\tilde{g}_{q\ov{q}}|\nu_{q}|^{2}+C_{A}K\right)$.
			
			\medskip
For a), by Lemma \ref{nu}, we deduce 
			\begin{equation}\label{case 2 inequality}
				\begin{split}
					\sum_{q\notin I}\tilde{g}_{q\ov{q}}|\nu_{q}|^{2}+C_{A}K
					\leq& \sum_{q}\tilde{g}_{q\ov{q}}|\nu_{q}|^{2}+C_{A}K 
					=\tilde{g}(W_1, \ov{W_{1}})+C_{A}K\\
					\leq& \frac{1}{2}\big(\lambda_1+\sum_{\alpha>1}\lambda_\alpha\mu_\alpha^2\big)+C_{A}K,
				\end{split}
			\end{equation}
			where we used \eqref{definition of W} in the last inequality. Combining this with the assumption of Case B, we see that
			\begin{equation}\label{case b inequality 1}
				\sum_{q\notin I}\tilde{g}_{q\ov{q}}|\nu_{q}|^{2}+C_{A}K \leq \frac{C_{A}K}{\ve}.
			\end{equation}
			Using Lemma \ref{nu} again and \eqref{case 2 inequality},
			\[
			\frac{1}{2}\Big(\lambda_1+\sum_{\alpha>1}\lambda_\alpha\mu_\alpha^2\Big)
			\geq \tilde{g}(W_1, \ov{W_{1}})-CK = \sum_{q}\tilde{g}_{q\ov{q}}|\nu_{q}|^{2}-CK \geq -C_{A}K,
			\]
		 which implies
			$
			0 < \lambda_{1}-\sum_{\alpha>1}\lambda_{\alpha}\mu_{\alpha}^{2}
			\leq 2\lambda_{1}+C_{A}K\leq (2+2\ve^{2})\lambda_{1}
			$
			under the  assumption $\lambda_{1}\geq\frac{C_{A}K}{\ve^{2}}$. Letting $\gamma=\ve^{-2}$, then
			\begin{equation*}\label{coeff1}
				\frac{1-\ve}{(2-\ve)\lambda_{1}}
				\Big(1+\frac{1}{\gamma}\Big)\Big(\lambda_{1}-\sum_{\alpha>1}\lambda_{\alpha}\mu_{\alpha}^{2}\Big)\leq \frac{2-2\ve}{2-\ve}(1+\ve^{2})^{2}.
			\end{equation*}
			Since $\ve=\frac{e^{Au(x_{0})}}{9}$, for a large $A$ we get
			$
			\frac{2-2\ve}{2-\ve}(1+\ve^{2})^{2} \leq 1.
			$
			This proves a). 
			
			\medskip
			For b), using \eqref{case b inequality 1} and $\gamma=\ve^{-2},$
			\begin{equation*}
				\Big(1-\frac{\ve}{2}\Big)\Big(\frac{1+\gamma}{\lambda_{1}}\Big)
				\Big(\sum_{q\notin I}\tilde{g}_{q\ov{q}}|\nu_{q}|^{2}+C_{A}K\Big)\leq \frac{C_{A}}{\ve^{3}\lambda_{1}}.
			\end{equation*}
			This proves b) provided by $\lambda_{1}\geq\frac{C_{A}}{\ve^{3}}$.
\end{itemize}
\end{proof}
		Consequently, the Proposition \ref{B 3} follows from   
		\eqref{E1}, \eqref{B33}  and \eqref{bad}.
	\end{proof}
	
	Now we are return to prove Case 3 of Theorem \ref{Thm4.1}.

	\begin{proof}[Proof of Case 3]
		Using Proposition \ref{lower bound of L Q} together with Lemma \ref{bad terms 1 2} and Proposition \ref{B 3}, we deduce
		\begin{equation*}
			\begin{split}
				0 \geq {} & (A^{2}e^{-Au}-9\ve A^{2}e^{-2Au}) G^{i\bar{i}}|u_{i}|^{2}-\frac{C}{\ve}\mathcal{G}\\
				&+ \frac{\psi'}{4}\sum_{j} G^{i\bar{i}}(|e_{i}e_{j}u|^{2}+|e_{i}\bar{e}_{j}u|^{2})
				-Ae^{-Au}L(u).
			\end{split}
		\end{equation*}
		Since $\ve=\frac{e^{Au(x_{0})}}{9}$, 
		\begin{equation}\label{proof of Thm4.1 eqn 1}
			\begin{split}
				0 \geq &-\frac{C}{\ve}\mathcal{G}+ \frac{\psi'}{4}\sum_{j} G^{i\bar{i}}(|e_{i}e_{j}u|^{2}+|e_{i}\bar{e}_{j}u|^{2})-Ae^{-Au}L(u).
			\end{split}
		\end{equation}
	Let $A=\frac{10C}{\theta}$, where $\theta$ is the constant given in Proposition \ref{prop subsolution}. There are two possibilities:
		
		\bigskip
		\noindent
		\begin{itemize}
			\item $-L(u)\geq \theta\mathcal{G}$. In this setting, \eqref{proof of Thm4.1 eqn 1} yields that
			\[
			0 \geq \Big(A\theta e^{-Au}-\frac{C}{\ve}\Big)\mathcal{G}+ \frac{\psi'}{4}\sum_{j} G^{i\bar{i}}(|e_{i}e_{j}u|^{2}+|e_{i}\bar{e}_{j}u|^{2}).
			\]
			Using the fact $A=\frac{10C}{\theta}$, we deduce
			\[
			A\theta e^{-Au}-\frac{C}{\ve}
			= A\theta e^{-Au}-9Ce^{-Au} = Ce^{-Au},
			\]
			which implies
			\[
			0\geq Ce^{-Au}\mathcal{G}+ \frac{\psi'}{4}\sum_{j} G^{i\bar{i}}(|e_{i}e_{j}u|^{2}+|e_{i}\bar{e}_{j}u|^{2})>0.
			\]
			This is impossible.
			
			\vspace{0.3cm}
			\item $G^{1\bar{1}}\geq \theta\mathcal{G}$. Using the Cauchy-Schwarz inequality,
			\begin{equation*}
				\begin{split}
					Ae^{-Au}L(u) = {} & Ae^{-Au}\sum_{i}G^{i\bar{i}}\big(e_{i}\ov{e}_{i}u-[e_{i},\ov{e}_{i}]^{(0,1)}u+e_{p}(u) Z^{p}_{i\bar i}+\bar{e}_{p}(u)\overline{Z^{p}_{i\bar i}}\big) \\
					\leq {} & Ae^{-Au}\mathcal{G}\sum_{i}|e_{i}\ov{e}_{i}u|+CAe^{-Au}K\mathcal{G} \\
					\leq {} & \frac{\theta\psi'}{8}\mathcal{G}\sum_{i}|e_{i}\ov{e}_{i}u|^{2}+C_{A}K\mathcal{G}.
				\end{split}
			\end{equation*}
			Plugging it into \eqref{proof of Thm4.1 eqn 1},
			\[
			\frac{\theta\psi'}{8}\mathcal{G}\sum_{i,j} (|e_{i}e_{j}u|^{2}+|e_{i}\bar{e}_{j}u|^{2})\leq C_{A} K\mathcal{G}
			\]
		and hence
			\[
			\sum_{i,j} (|e_{i}e_{j}u|^{2}+|e_{i}\bar{e}_{j}u|^{2})\leq C_{A} K^{2}.
			\]
		This yields $\lambda_{1}\leq C_{A}K$ and the proof is completely.
		\end{itemize}
	\end{proof}
	Now we give the proof of Theorem \ref{main estimate}.
	\begin{proof}
		Combining Proposition \ref{Prop3.2} and Theorem \ref{HMW}, we obtain Theorem \ref{main estimate}.
	\end{proof}
	
	\subsection{Higher order estimates}
	\begin{proposition}\label{main estimate 1}
		Let $(M,\chi,J)$ be a compact almost Hermitian manifold of real dimension $2n$. Suppose $f$ satisfies $\mathrm{(i)}$,  $\mathrm{(ii)}$ and $\mathrm{(iii')}$ on a symmetric open and convex cone $\Gamma \subsetneq\mathbb{R}^{n}$ as in \eqref{assumption of cone}.
		Assume	$\underline{u}$ is a $\mathcal{C}$-subsolution  and $u$ is a smooth solution of \eqref{nonlinear equation}. Then for each $k=0,1,2,\cdots$, we have 
		\begin{equation*}
			\| u\|_{C^{k}(M,\chi)}\leq C_{k},
		\end{equation*}
		where $C_{k}$ is a constant depending on $k$, $\underline{u}$, $h$, $Z$, $\omega$, 
		$f$, $\Gamma$ and $(M,\chi,J)$.
	\end{proposition}
	
	\begin{proof}[Proof of Proposition \ref{main estimate 1}]
		With the estimate \eqref{HMW} at hand, a standard blow-up argument \cite[Proposition 5.1]{CHZ21} combining with Liouville theorem \cite[Theorem 20]{Szekelyhidi18} (see also \cite{DK17,TW17,TW19,STW17}), we conclude $\sup_{M}|\partial u|\leq C$. Although the appearance of the term $Z$ which depends on $\partial u$  linearly, it does not matter under the rescaling procedure.  The more details can be found in \cite[\S 5]{CHZ21}. 
		
		\medskip
		We can then apply the Evans-Krylov-type estimate (see \cite[Theorem 1.1]{TWWY15} and \cite[\S 5]{CHZ21}). The higher estimates can be obtained by applying a standard bootstrapping
		argument, we shall omit the standard step here. 
	\end{proof}
	\subsection{Proof of Theorems  \ref{n-1 MA equation}-\ref{complex Hessian equation}}
	We remark that equation \eqref{sove n-1 MA equation} and equation \eqref{complex Hessian equation} satisfying the structural conditions (i), (ii) and (iii'). Using Proposition \ref{main estimate 1} and a similar arguments in the proof of \cite[Theorem 1.1]{CTW19} and \cite[Theorems 1.2-1.3]{CHZ21}, we obtain Theorems \ref{n-1 MA equation}-\ref{complex Hessian equation}. 
	\qed

	\section{Proofs of Corollary \ref{DHYM estimates}}\label{applications}

	In this section, we prove Corollary \ref{DHYM estimates}. First, we give the $C^{1}$ estimates of the dHYM equation \eqref{DHYM2}.
	
	\begin{prop}\label{C1 DHYM}
		Let $u$ (resp. $\underline{u}$) be the solution (resp. $\mathcal{C}$-subsolution) for \eqref{DHYM2} with $\sup_{M}(u-\underline{u})=0$. Then we have
		\begin{equation*}\label{}
			\|u\|_{C^{1}}\leq C,
		\end{equation*}
		where $C$ depending on $\underline{u}$, $h$, $\|\omega\|_{C^{1}}$,  $\Gamma$ and $(M,\chi,J)$.
	\end{prop}
	
	\begin{proof}
		Let us define \[H(\eta)=\frac{1}{3}e^{D\eta},\qquad \eta=\underline{u}-u.\] Here $D>0$ are certain constants to be picked up later.\footnote{From now on, the $C$ below denotes the constants those may change from line to line, and it doesn't depend on $D$ that we yet to choose.} Consider the test function
		\[Q= e^{H(\eta)}|\partial  u|^{2}.\]
		Suppose $Q$ achieves maximum at the $x_{0}\in M$. We may assume 
		$|\partial u|(x_{0})\geq 1.$ Otherwise we are done.
		Then near $x_{0}$, we can choose a proper local frame $\{e_{i}\}_{i=1}^{n}$ such that $\chi_{i\bar{j}}=\delta_{ij}$ and the matrix $\big\{\tilde{g}_{i\bar{j}}\big\}$ is diagonal at $x_{0}$. It follows from maximum principle that
		\begin{equation}\label{3.9}
			\begin{split}
				0\geq \frac{L(Q)(x_{0})}{DH e^{H}|\partial  u|^{2}}
				=&L(\eta)+D(1+H)G^{i\bar{i}}|\eta_{i}|^{2}+\frac{L(|\partial  u|^{2})}{DH|\partial  u|^{2}}\\
				&+\frac{2}{|\partial  u|^{2}}\sum_{i,j}G^{i\bar{i}}\textrm{Re}\big\{e_{i}(\eta)\bar{e}_{i}e_{j}(u)\bar{e}_{j}(u)
				+e_{i}(\eta)\bar{e}_{i}\bar{e}_{j}(u)e_{j}(u)\big\}.
			\end{split}
		\end{equation}
		By a similar argument to Lemma \ref{lower bound of L de u}, we get
		\begin{lemma}\label{lower bound of L de u DHYM}
			At $x_{0}$, we have, for every $\ve\in(0, \frac{1}{2})$,
			\begin{equation*}
				L(|\partial u|^{2}) \geq (1-\ve) \sum_{j}G^{i\bar{i}}(|e_{i}e_{j}u|^{2}+|e_{i}\bar{e}_{j}u|^{2})-\frac{C}{\ve}|\partial u|^{2}\mathcal{G}.
			\end{equation*}
		\end{lemma}
		Dividing by $DH|\partial  u|^{2}$, we have
		\begin{equation}\label{3.7C1}
			\begin{split}
				\frac{L(|\partial  u|^{2})}{DH|\partial  u|^{2}}\geq&  (1-\varepsilon) \sum_{i,j}G^{i\bar{i}}\frac{|e_{i}e_{j}u|^{2}+|e_{i}\bar{e}_{j}u|^{2}}{DH|\partial  u|^{2}}- \frac{C\mathcal{G}}{DH\varepsilon}.
			\end{split}
		\end{equation}
		For the last term of (\ref{3.9}). Note that $\ve\in (0,\frac{1}{2}]$ implies $1\leq (1-\varepsilon)(1+2\varepsilon)$. Using the definition of Lie bracket again,  we see
		\begin{equation}\label{3..9}
			\begin{split}
				2\sum_{i,j}&G^{i\bar{i}}\textrm{Re}\{e_{i}(\eta)
				\bar{e}_{i}e_{j}(u)\bar{e}_{j}(u)\}\\
				\stackrel{}{=}& 2\sum_{i,j}G^{i\bar{i}}\textrm{Re}\Big\{\eta_{i}
				u_{\bar{j}}\big\{e_{j}\bar{e}_{i}(u)-[e_{j},\bar{e}_{i}]^{0,1}(u)
				-[e_{j},\bar{e}_{i}]^{1,0}(u)\big\}\Big\}\\
				=&  2\sum_{i}G^{i\bar{i}}(\mu_{i}-g_{i\bar{i}})\textrm{Re}\big\{\eta_{i}
				u_{\bar{i}}\big\}-2\sum_{i,j}G^{i\bar{i}}\textrm{Re}\big\{\eta_{i}
				u_{\bar{j}}[e_{j},\bar{e}_{i}]^{1,0}(u)\big\}\\
				\geq&  2\sum_{i}G^{i\bar{i}}(\mu_{i}-g_{i\bar{i}})\textrm{Re}\{\eta_{i}
				u_{\bar{i}}\}-\varepsilon DH|\partial  u|^{2}\sum_{i}G^{i\bar{i}}|\eta_{i}|^{2}
				-\frac{C}{DH\varepsilon}|\partial  u|^{2}\mathcal{G}
			\end{split}
		\end{equation}
		and
		\begin{equation}\label{3..10}
			\begin{split}
				2\sum_{i,j}&G^{i\bar{i}}\textrm{Re}\big\{e_{i}(\eta)\bar{e}_{i}\bar{e}_{j}(u)e_{j}(u)\big\} \\
				\geq & -\frac{(1-\varepsilon)}{DH}\sum_{i,j}G^{i\bar{i}}|\bar{e}_{i}\bar{e}_{j}(u)|^{2}
				-(1+2\varepsilon)DH|\partial  u|^{2}\sum_{i}G^{i\bar{i}}|\eta_{i}|^{2}.
			\end{split}
		\end{equation}
		It follows from (\ref{3..9}) and (\ref{3..10}) that
		\begin{equation}\label{3.12}
			\begin{split}
				 \frac{2}{|\partial  u|^{2}}\sum_{i,j}&G^{i\bar{i}}\textrm{Re}\big\{e_{i}(\eta)\bar{e}_{i}e_{j}(u)\bar{e}_{j}(u)
				+e_{i}(\eta)\bar{e}_{i}\bar{e}_{j}(u)e_{j}(u)\big\}\\
				\geq &\frac{2}{|\partial  u|^{2}}\sum_{i}G^{i\bar{i}}(\mu_{i}-g_{i\bar i})\textrm{Re}\{e_{i}(\eta)
				\bar{e}_{i}(u)\}-\frac{C\mathcal{G}}{DH\varepsilon}\\&
				-(1+3\ve)DH \sum_{i}G^{i\bar{i}}|\eta_{i}|^{2}
				-(1-\varepsilon)\sum_{i,j}G^{i\bar{i}}\frac{|\bar{e}_{i}\bar{e}_{j}(u)|^{2}}{DH|\partial  u|^{2}}.
			\end{split}
		\end{equation}
		Combining (\ref{3.9}), (\ref{3.7C1}) and (\ref{3.12}), and letting $\ve=\frac{1}{6H(x_{0})}$,
		\begin{equation*}
			\begin{split}
				&L(\eta)+ \frac{2}{|\partial  u|^{2}}\sum_{i}G^{i\bar{i}}(\mu_{i}-g_{i\bar{i}})\textrm{Re}\{\eta_{i}
				u_{\bar{i}}\}
				+\frac{D}{2}
				\sum_{i}G^{i\bar{i}}|\eta_{i}|^{2}
				\leq \frac{C}{DH|\partial  u|}+ \frac{C\mathcal{G}}{D}.
			\end{split}
		\end{equation*}
		By the assumption $|\partial  u|\geq \max\{1,|\partial  \underline{u}|\}$, we obtain
		\[
		\begin{split}
			\frac{2}{|\partial  u|^{2}}\sum_{i}&G^{i\bar{i}}(\mu_{i}-g_{i\bar{i}})\textrm{Re}\{\eta_{i}
			u_{\bar{i}}\} \\ \geq &   -\frac{D}{4}\sum_{i}G^{i\bar{i}}|\eta_{i}|^{2}
			-\frac{C}{D|\partial  u|^{2}}\sum_{i}\frac{(\mu_{i}-1)^{2}}{1+\mu_{i}^{2}}\\
			\geq &   -\frac{D}{4}\sum_{i}G^{i\bar{i}}|\eta_{i}|^{2}
			-\frac{C}{D|\partial  u|^{2}}.
		\end{split}
		\]
		Hence,
		\begin{equation}\label{3.11}
			\begin{split}
				&L(\eta)
				+\frac{D}{4}
				\sum_{i}G^{i\bar{i}}|\eta_{i}|^{2}
				\leq \frac{C}{DH|\partial  u|}+ \frac{C\mathcal{G}}{D}+\frac{C}{D|\partial  u|^{2}}.
			\end{split}
		\end{equation}
		There are two possibilities:
		\begin{itemize}
			\item If \eqref{property 1 of sub} holds.
			It follows from (\ref{3.11}) that
			\[
			\theta
			+\theta\mathcal{G}\leq  \frac{C}{DH|\partial  u|}+ \frac{C\mathcal{G}}{D}+\frac{C}{D|\partial  u|^{2}}.
			\]
			Choose $D$ large such that $\theta>\frac{C}{D}$. 
			Then we get
			\begin{equation*}\label{3.14}
				\theta\leq \frac{C}{DH|\partial  u|}+\frac{C}{D|\partial  u|^{2}}.
			\end{equation*}
			This implies $|\partial  u|\leq C$.

			\vspace{0.3cm}
			\item If \eqref{property 2 of sub} is true.  By \eqref{lower bound of F}, we have
			$ G^{1\bar{1}}\geq \theta\mathcal{G}\geq \theta\tau. $ Therefore,
			\[
			\sum_{i}G^{i\bar{i}}|\eta_{i}|^{2}\geq \theta\tau|\partial  \eta|^{2},
			\]
			and 
			\begin{equation*}
				L(\eta)=G^{i\bar i}((g_{i\bar i}+\underline{u}_{i\bar i})- \mu_{i})\geq -C-C\sum_{i}\frac{|\mu_{i}|}{1+\mu_{i}^{2}}\geq -C.
			\end{equation*}
			Plugging the above two inequalities into (\ref{3.11}),
			\[
			\frac{D}{C}|\partial  \eta|^{2}\leq \frac{C}{DH|\partial  u|}+C.
			\]
			We may assume that $|\partial u|\geq 2|\partial \underline{u}|$ and then $|\partial \eta|\geq \frac{1}{2}|\partial u|$. So
			\[
			\frac{D}{C}|\partial u|^{2}\leq \frac{C}{DH|\partial  u|}+C.
			\]
			As a consequence, $|\partial u|\leq C$.
		\end{itemize}
		
		\medskip
		Combining the Theorem \ref{main estimate}, we establish the second order estimates. Therefore, the equation \eqref{DHYM} is uniform elliptic. Based on Evans-Krylov theory, we obtain the higher order estimates. This completes the proof of Corollary \ref{DHYM estimates}.
	\end{proof}

\end{document}